%% file: random-walks-on-randomly-evolving-graphs-arxiv.tex
\documentclass[11pt]{article}
\usepackage[letterpaper, margin=1in ]{geometry}

\usepackage[utf8]{inputenc}
\usepackage{enumerate}
\usepackage{url}
\usepackage{cite}
\usepackage{hyperref}

\usepackage{amsmath,amssymb,amsthm}
\usepackage[utf8]{inputenc}
\numberwithin{equation}{section}
\usepackage{mathrsfs}
\usepackage{bbm}
\usepackage{thmtools}
\usepackage{thm-restate}
\usepackage[ruled,vlined]{algorithm2e}
\usepackage[noend]{algpseudocode}
\usepackage{bm}
\usepackage{array}

\usepackage{pgf,tikz,pgfplots}
\usetikzlibrary{arrows,chains,matrix,positioning,scopes}
\usepackage{tcolorbox}
\usepackage{wrapfig}
\usepackage{float}
\usepackage{twoopt}
\usepackage{multirow}
\usepackage{booktabs}
\usepackage{chngcntr}
\counterwithin{table}{section}

\newcommand{\nor}[2]{\bigl\| #1 \bigr\|_{#2}}
\input{notationmacrosarxiv}

\usepackage{graphicx}
\begin{document}
	\title{Random walks on randomly evolving graphs\thanks{The second and third author acknowledge support by the ERC Starting Grant ``Dynamic March''.}}
	\author{Leran Cai \and Thomas Sauerwald \and Luca Zanetti}
	%
	%\authorrunning{L. Cai et al.}
	%\institute{University of Cambridge, Cambridge, UK \\
	%	\email{\{lc647,tms41,lz381\}@cam.ac.uk}\\
	%}
	%
	\maketitle              % typeset the header of the contribution
	\begin{abstract}
		A random walk is a basic stochastic process on graphs and a key primitive in the design of distributed algorithms. One of the most important features of random walks is that, under mild conditions, they converge to a stationary distribution in time that is at most polynomial in the size of the graph. This fundamental property, however, only holds if the graph does not change over time, while on the other hand many distributed networks are inherently dynamic, and their topology is subjected to potentially drastic changes. 
		
		In this work we study the mixing (i.e., converging) properties of random walks on graphs subjected to random changes over time. Specifically, we consider the edge-Markovian random graph model: for each edge slot, there is a two-state Markov chain with transition probabilities $p$ (add a non-existing edge) and $q$ (remove an existing edge). We derive several positive and negative results that depend on both the density of the graph and the speed by which the graph changes. 
		We show that if $p$ is very small (i.e., below the connectivity threshold of Erd\H{o}s-R\'{e}nyi random graphs), random walks do not mix (fast). When $p$ is larger, instead, we observe the following behavior: if the graph changes slowly over time (i.e., $q$ is small), random walks enjoy strong mixing properties that are comparable to the ones possessed by random walks on static graphs; however, if the graph changes too fast (i.e., $q$ is large), only coarse mixing properties are preserved.

		%The abstract should briefly summarize the contents of the paper in
		%15--250 words.
		
		%\keywords{Random Walks  \and Evolving Graphs \and Mixing Times}
	\end{abstract}

	\section{Introduction}
	
	A random walk on a network is a simple stochastic process, defined as follows. Given an undirected graph $G=(V,E)$, the walk starts at a fixed vertex. Then at each step, the random walk moves to a randomly chosen neighbor \footnote{In case of a {\em lazy} random walk, the walk would remain at the current location with probability $1/2$, and otherwise move to a neighbor chosen uniformly at random.}.  Due to their simplicity and locality, random walks are very useful algorithmic primitive, especially in the design of distributed algorithms. In contrast to topology-driven algorithms, algorithms based on random walks benefit from a strong robustness against structural changes in the network.
	
	Random walks and related works have found various applications such as routing, information spreading, opinion dynamics and graph exploration~\cite{DBLP:conf/sirocco/Cooper11,ChenDynamicRSA}. One key property of random walks is that, under mild assumptions on the underlying network, they converge to a stationary distribution -- an equilibrium state in which every vertex is visited proportionally to its degree. The time for this convergence to happen is called {\em mixing time}, and understanding this time is crucial for many sampling or exploration related tasks. In particular, whenever a graph has a small {\em mixing time}, also its {\em cover time} (the expected time to visit all vertices of the graph) is small as well.
	
	While most of the classical work devoted to understanding random walks has focused on static graphs, many networks today are subject to dramatic changes over time. Hence understanding the theoretical power and limitations of dynamic graphs has been identified as one of the key challenges in computer science~\cite{MS18}. Several recent works have indeed considered this problem and studied the behavior of random walks \cite{AugustinePR16,ChenDynamicRSA,SMP15,DenysyukR14,SpirakisCover18,merging1,merging2} or similar processes \cite{icalp2016voter,Clementi2016,ClementiST15,GiakkoupisSS14,KuhnO11} on such dynamic graphs, and their applications to distributed computing~\cite{AugustinePR16,SMP15,KuhnO11}. 
	Moreover, rather than a property of the underlying network itself, dynamic graphs may naturally arise in distributed algorithms when communication is performed on a changing, possibly disconnected, subgraph like a spanning-tree or a matching (see, e.g., \cite{gossip}).
%	\NOTE{T}{Dynamic graphs also arise from static graphs if communication is restricted, e.g., by matchings.}
	
	In this work, we study the popular {\em evolving graph model}. That is, we consider sequences of graphs $G_1, G_2,\ldots$ over the same set of vertices but with a varying set of edges. This model has been studied in, for example, \cite{ChenDynamicRSA,SauerwaldZanetti,SpirakisCover18}. Both \cite{ChenDynamicRSA} and later \cite{SauerwaldZanetti} proved a collection of positive and negative results about the mixing time (and related parameters), and they assume a worst-case scenario where the changes to the graph are dictated by an oblivious, non-adaptive adversary. For example, \cite{ChenDynamicRSA} proved the following remarkable dichotomy. First, even if all graphs $G_1,G_2,\ldots$ are connected, small (but adversarial) changes to the stationary distribution can cause exponential mixing (and hitting) times. Secondly, if the sequence of connected graphs share the same stationary distribution i.e., the degrees (or relative degrees) of vertices are time-invariant, then mixing and hitting times are polynomial. This assumption about a time-invariant stationary distribution is crucial in the majority of the positive results in \cite{ChenDynamicRSA,SauerwaldZanetti}.
	In general, in order to compensate for the adversarial changes, both works had to impose assumptions about the existence of a time-invariant stationary distribution, which means that the degree (or relative degree) of a vertex remains the same at each step. Both \cite{ChenDynamicRSA} and later \cite{SauerwaldZanetti} proved a collection of positive and negative results about the mixing time (and related parameters), assuming that the graph changes {\em arbitrarily} (i.e., by an oblivious, non-adaptive adversary). However, to compensate for the fact that the graph changes in an arbitrary way, both works had to impose assumptions about connectivity and the existence of a time-invariant stationary distribution, which means that the degree (or relative degree) of a vertex remains the same at each step.
	%\NOTE{Luca}{We first say they let the graph changes arbitrarily, and then they impose strong conditions on it. I understand what it was supposed to mean, but might be a bit confusing.} 
	
	In contrast to \cite{ChenDynamicRSA,SauerwaldZanetti}, we do not impose such assumptions, but instead study a model with incremental changes. Specifically, we consider a setting where the evolving graph model changes {\em randomly} and study the so-called {\em edge-Markovian random graph} $\cg(n,p,q)$, which is defined as follows (see Definition~\ref{def:edgemarkovian} for a more formal description). For each edge slot, there is a two-state Markov chain that switches from off to on with probability $p$ and from on to off with probability $q$. This model can be seen as a dynamic version of the Erd\H{o}s-R\'{e}nyi random graph, and has been studied in the context of information spreading and flooding \cite{Clementi2016,Clementi2011,ClementiMMPS10}. While these results demonstrate that information disseminates very quickly on these dynamic graphs, analysing the convergence properties of a random walk seems to require new techniques, since the degree fluctuations make the use of any ``inductive'' arguments very difficult -- from one step to another, the distribution of the walk could become worse, whereas the set of informed (or reachable) nodes can never decrease.
	
	In this work, we will investigate the mixing time of a random walk on such evolving graphs. It turns out that, as our results demonstrate, the mixing time depends crucially on the density as well as on the speed by which the graph changes. We remark that deriving bounds on the mixing time on $\cg(n,p,q)$ poses some unique challenges, which are not present in the positive results of \cite{ChenDynamicRSA,SauerwaldZanetti}. 
	The main difficulty is that in $\cg(n,p,q)$, due to the changing degrees of the vertices, there is no time-invariant stationary distribution, and the traditional notion of mixing time must be adapted to our dynamic setting. Informally, what we ask, then, is how many steps the walk needs to take before the distance to a \emph{time-dependent} stationary distribution becomes \emph{small enough}. Furthermore, in contrast to static graphs, where the distance between the distribution of the walk and the stationary distribution can only decrease, in dynamic graphs the distance to the time-dependent stationary distribution might increase with time. For this reason, we also ask that the distribution of the walk remains close to a time-dependent stationary distribution for a \emph{long enough} interval of time (for a precise definition of our notion of mixing time, see Definition~\ref{def:dynamicmixingtime}). We believe this requirement is necessary for our definition of mixing time to be useful in potential applications.
	
	\subsubsection*{Further Related Work.}
	Recently,  \cite{SpirakisCover18} analysed the cover time of so-called ``Edge-Uniform Stochastically-Evolving Graphs'', that include our model as a special case (i.e., the history is $k=1$). Their focus is on a process called ``Random Walk with a Delay'', where at each step, the walk picks a (possible) neighbor and then waits until the edge becomes present. In \cite[Theorem 4]{SpirakisCover18}, the authors also relate this process to the standard random walk, and prove a worst-case upper bound on the cover time. However, one of the key difference to \cite{SpirakisCover18} is that we will study the {\em mixing time} instead of the {\em cover time}. 
	
	In \cite{sousiThomas}, the authors~analysed a continuous-time version of the edge-Markovian random graph. However, unlike the standard random walk, they consider a slightly different process: when the random walk tries to make a transition from a vertex $u$, it picks one of the $n-1$ other vertices and moves there {\em only if} the edge is present; otherwise it remains in place. For this process, they were able to derive very tight bounds on the mixing time that establish the so-called cutoff phenomena. The same random walk was also analysed on a dynamic graph model of the $d$-dimensional grid in \cite{pss15,PeresSousiSteif}, and, more generally, in \cite{hermonSousi}.

	\subsection{Main Results}
	We study the mixing properties of random walks on edge-Markovian random graphs $\cg(n,p,q)$. % (Def. \ref{def:edgemarkovian}) \cite{Clementi2011, Clementi2016}.
	In particular, we consider six different settings of parameters $p$ and $q$, which separates edge-Markovian models based on how fast graphs change over time (slowly vs. fast changing), and how dense graphs in the dynamic sequence are (sparse vs. semi-sparse vs dense).

	As noted in previous works (see, e.g., \cite{Clementi2016}), a dynamic sequence sampled from $\cg(n,p,q)$ will eventually converge to an \erdosrenyi random graph $\cg(n, \tilde{p})$ where $\tilde{p} = \frac{p}{p+q}$  (for the sake of completeness, we give a proof in Appendix $\ref{sec:graphchain}$). We use the expected degree in such a random graph, which is equal to $d=(n-1)\tilde{p}$, to separate edge-Markovian models according to their density as follows:
	\begin{enumerate}
		\item \textbf{Sparse} $d = o(\log{n})$
		\item \textbf{Semi-sparse} $d=\Theta(\log{n})$
		\item \textbf{Dense} $d=\omega(\log{n})$.
	\end{enumerate}
	Notice that the sparse regime corresponds to random graphs with density below the connectivity threshold of \erdosrenyi random graphs.
	
	We further separate edge-Markovian models based on how fast they change over time. Let $\delta = \binom{n}{2}\tilde{p}q + \binom{n}{2}(1 - \tilde{p})p$ be the expected number of changes at each step, when starting from a stationary initial graph $G_0 \sim \cg(n,\tilde{p})$. We consider the following two opposite regimes.
	\begin{enumerate}
		\item \textbf{Fast-changing} $\delta = \Theta(dn)$.
		\item \textbf{Slowly-changing} $\delta = O(\log{n})$
	\end{enumerate}
	Notice that the fast-changing regime corresponds to graphs for which a constant fraction of edges change at each step in expectation.
	
	For a cleaner exposition, we consider six different settings based on fast-changing vs. slowly-changing rates and dense vs. medium vs. sparse graphs (Table \ref{tab:results}). The density of the graph refers to the expected degree of the graph sampled from the {\em stationary graph distribution} (Sec.~\ref{sec:notations}) of the model. Moreover this quantifies the expected degree $d = n\tilde{p}$. For reasons of space, we will not analyse all combinations of $p$ and $q$. Also note that some choices of $p$ and $q$ are not interesting, for example, small values of $p$ (and $q$) may trap the walk inside a small region for a long time.
	
	% \begin{enumerate}
	%     %\item Independent cases: $p + q = 1$.
	%     %\item Dependent cases: $p + q \ne 1$. \NOTE{T}{I guess we are usually considering a ``highly dependent'' case where $p+q$ are bounded from below by a constant $<1$ right?}
	%     %\begin{enumerate}
	%         \item Fast-changing: at each step, the expected number of edge changes is $\Omega(n d)$. 
	%         \item Slowly-changing: at each step, the expected number of edge changes is $O(\log n)$. 
	%         \item Sparse: The expected degree $d$ is between a positive constant and $o(\log n)$.
	%         \item Medium: The expected degree $d$ is $\Theta(\log n)$.
	%         \item Dense: The expected degree $d$ of the graph is between $\omega(\log n)$ and $n/2$. 
	%     %\end{enumerate}
	% \end{enumerate}
	
	% \textbf{Choosing $p$ and $q$.} For each of the six possible combinations, we first fix a desired average degree $d$ of the limiting graph. Then, we either choose $q=\Omega(1) > 0$ (in case of fast-changing) or $q=O(\log n / (n \cdot d))$ (in case of slowly-changing). Since $d=(n-1) \cdot \frac{p}{p+q}$, once $q$ and $d$ are fixed, we choose $p$ so that above equation holds. 
	
	% In case of the fast-changing model, the initial graph $G_0$ can be arbitrary, though it is equivalent to sampling $G_0$ from the stationary distribution $\cg(n, \tilde{p})$ since it converges fast. For our positive result in the slowly-changing model, we make more specific assumptions about $G_0$.
	
	\setlength{\tabcolsep}{4pt}
	\renewcommand*{\arraystretch}{1.1}
	\begin{table}[h]\label{tab:results}
		\centering
		\begin{tabular}{c c c}
			\toprule
			& \textbf{Fast-changing}     & \textbf{Slowly-changing} \\ 
			& \footnotesize{$\delta = \Theta(dn)$}  & \footnotesize{$\delta = O(\log{n})$} \\
			\hline
			\textbf{Sparse}     &  $\tmix = \infty$    &  $\tmix = \Omega(n)$   \\ 
			\footnotesize{$d \in [ 1,o(\log n)]$} & Thm~\ref{thm:fastsparse} & Proposition~\ref{prop:slowsparse}  \\ \hline
			\textbf{Semi-sparse}      &  Coarse mixing\footnotemark~in $O(\log n)$   &  \\
			\footnotesize{$d = \Theta(\log n)$} & Prop~\ref{prop:coarse}  & $\tmix = O(\log n)$,  \\
			\cline{1-2}
			\textbf{Dense}      &  $\tmix = O(\log n)$     & Thm~\ref{thm:slowdense}  \\
			\footnotesize{$d \in [ \omega(\log n),n/2]$} & Thm~\ref{thm:fastdense} &   \\
			\toprule
		\end{tabular}
		\vspace{0.5em}
		\caption{Summary of our main results (informal). See referenced theorems for the precise and complete statements.}
	\end{table} 
	\footnotetext{In this regime we are not able to prove finite mixing time. However, we show that the distribution of the walk will ``flatten out'' after $O(\log{n})$ steps. We refer to this behavior as \emph{coarse} mixing.}
	
	The main results of our work are presented in Table~\ref{tab:results}. Here, we assume $G_0$ is sampled from the stationary graph distribution $\cg(n, \tilde{p})$. In the fast-changing regime, as highlighted in Remark \ref{rem:graphchain}, this is without loss of generality. For slow-changing models, instead, different choices of $G_0$ can result in drastically different outcomes with regard to the mixing time. For ease of presentation, we assume in Table~\ref{tab:results} that $G_0 \sim \cg(n, \tilde{p})$, but this assumption can usually be relaxed, and we refer to the full statement of the corresponding results for our actual assumptions on $G_0$.
	
	Next we formally state the four main results of our work. The formal definitions of mixing time for random walk on dynamic graphs will be presented in \secref{notations} (see in particular Definition \ref{def:dynamicmixingtime} and Definition \ref{def:edgemarkovianmixingtime}). The first theorem is a negative result that tells us  that, for fast-changing and sparse edge-Markovian graphs, random walks don't have finite mixing time. Its proof will be presented in \secref{fastsparse}.
	
	%In non-sparse cases, this means $G_0$ is an expander with probability at least $1 - n^{-\Omega(1)}$. However for our strong mixing result in the slowly-changing and non-sparse regime, we actually prove a stronger theorem: by only making assumptions weaker than assuming $G_0$ is a expander, we can still get the same mixing time \NOTE{T}{same as the graph $G_0$ or same as an expander?} for the random walk process. Note that our proof will use the fact that the graph process converges to its stationary distribution, however, the random walk process mixes much faster than the graph process.
	
	\begin{restatable}[Fast-changing and sparse, no mixing]{theorem}{fastsparse}\label{thm:fastsparse}
		Let $p = \Theta \left(\frac{1}{n} \right)$ and $q = \Omega(1)$. Then, $\tmix(\cg(n,p,q)) = \infty$.
	\end{restatable}
	
	The following theorem is, instead, a positive result that establishes fast mixing time in the dense and fast-changing regime. Its proof is presented in \secref{fastdense}. 
	
	\begin{restatable}[Fast-changing and dense, fast mixing]{theorem}{fastdense}\label{thm:fastdense}
		Let $p = \omega\left(\log{n}/n\right)$ and $q = \Omega(1)$. Then,
		$\tmix(\cg(n,p,q)) = O(\log{n})$.
	\end{restatable}
	
	The only case missing in the fast-changing regime is the semi-sparse case, where nodes have average degree $d = \Theta(\log{n})$. We do not have a definitive answer on the mixing time of random walks in such case, however, we do have a partial result that guarantees at least that random walk distributions will be ``well spread'' over a large support after $O(\log{n})$ steps (we call this behavior \emph{coarse mixing}). This statement can be made formal by considering the $\ell_2$-norm of the distribution of the walk. Because of its technical nature, we defer the formal statement to Proposition~\ref{prop:coarse}.
	
	We now turn our attention to the slowly-changing regime, where at most $\delta = O(\log n)$ edges are created and removed at each step. Unlike the results for the fast-changing regime,  where the choice of the starting graph $G_0$ does not really affect the mixing time of a random walk (see Appendix \ref{sec:graphchain} and Remark \ref{rem:graphchain} for a discussion), in the slowly-changing regime the choice of $G_0$ will affect the properties of $G_t$ for a large number of steps $t$.
	
	The following theorem shows that in the slowly-changing and dense regime, under mild conditions on the starting graph $G_0=(V,E_0)$ (which are satisfied for $G_0$ drawn from the limiting distribution of dense $\cg(n,p,q)$), random walks will mix relatively fast. We use $E_0(S,V\setminus S)$ to indicate the set of edges in $G_0$ between a subset of vertices $S \subset V$ and its complement, and $\Phi_{G_0}$to indicate the minimum conductance of $G_0$ (see Definition \ref{def:conductance}).
	
	%for the strong mixing time result (Thm.~\ref{thm:slowdense}), we assume $G_0$ only expands well on small sets instead of assuming it is a good expander. As mentioned above, we let $\tilde{p} = \frac{p}{p+q}$ and the expected degree $d = n\tilde{p}$ in the stationary graph distribution. 
	
	\begin{restatable}[Slowly-changing and dense, fast mixing]{theorem}{slowdense}\label{thm:slowdense}
		Let $d = \Omega(\log{n})$, $p=O(\log n/n^2)$, and $q=O(\log n/(dn))$. Let the following assumptions on the starting graph $G_0 = (V,E_0)$ be satisfied for large enough constants $c_1, c_2,c_3>0$.
		\begin{enumerate}[(1)]
			\item $\deg_0(x) = \Theta(d)$ for any $x \in V$;
			\item $|E_0(S,V\setminus S)| \ge c_2 \log{n} |S|$, for any $S \subset V$ with $|S| \le c_1 \log n$;
			\item $\Phi_{G_0} \ge c_3 \log d/d$.%, for any $S \subset V$ with $|S| > c_1 \log n$
		\end{enumerate}
		Then, $\tmix(\cg(n,p,q)) = O(\log n/\Phi_{G_0}^2)$.
		% Let $\cg(n,p,q)$ be an edge-Markovian model in the slowly-changing and dense regime. For the initial graph $G_0$, we assume the following conditions: the density  for any subset $S \subseteq V$ such that $|S| = O(\log n)$, the conductance of $S$, $\Phi_G(S)$ is a constant; for $|S| = \omega(\log n)$, $\Phi_G(S) = \Omega(\log d/d)$. Then,
		% \[
		% \tmix(\cg(n,p,q)) = O(\log n).
		% \]
	\end{restatable}
	Let us briefly discuss the assumptions and results of \thmref{slowdense}.First of all notice that the parameters $p$ and $q$ are defined so that the average degree $d = \Omega(\log(n))$ and the number of changes in the graph at each step is $\delta = O(\log(n))$. Assumption (1) just require the degree of the vertices in $G_0$ to be of the same order as the degree of the vertices in the limiting graph $\cg(n,\tilde{p})$. Assumption (2) guarantees that for any small set $S$ there are enough edges going from $S$ to the rest of the graph. Assumption (3) is a mild condition on the conductance of $G_0$. These two conditions ensure that the conductance of $G_t$ will not be much lower than the conductance of $G_0$ for a large number of steps $t$. Finally, notice that $O(\log n/\Phi_{G_0}^2)$ is a classic bound for the mixing time of a \emph{static} random walk on $G_0$.
	\thmref{slowdense} essentially states that, if the three assumptions are satisfied, the mixing time of a random walk on $\cg(n,p,q)$ will not be much larger. In particular, all the three assumptions are satisfied for a starting graph $G_0 \sim \cg(n,\tilde{p})$ with $\tilde{p}=p/(p+q)$. Furthermore, in such case $\tmix(\cg(n,p,q)) = O(\log n)$. The proof of this theorem can be found in \secref{slowdense}.
	
	We conclude this section by stating our result in the slowly-changing and dense regime. We prove a negative result: we show that the mixing time of $\cg(n,p,q)$ is at least linear in $n$. 
	
	\begin{restatable}[Slowly-changing and sparse, slow mixing]{proposition}{slowsparse}\label{prop:slowsparse}
		Let  $p = O(1/n^2)$ and $q = \omega(1/(n \log n))$. Consider a random walk on $\cg(n,p,q)$ with starting graph $G_0 \sim \cg(n,\tilde{p})$ with $\tilde{p} = p/(p+q)$. Then, $\tmix(\cg(n,p,q))=\Omega(n)$.
		% \[
		% \tmix(\cg(n,p,q))=\Omega(n).
		% \]
	\end{restatable}

	\section{Notation and Definitions}
	
	\subsection{Random Walk and Conductance}\label{sec:notations}
	%Let $G=(V, E)$ be an undirected graph, where $V$ is the set of nodes and $E$ is the set of edges. Random walks on static connected graphs are essentially ergodic Markov chains. A (lazy) random walk process is $(X_t)_{t \in \Nat}$ where $X_t$ is the position of the walk at time $t$. We define an $n$-dimensional row vector $\mu_t$ where $\mu_t(u)$ is the probability that the walk is at $u$ at time $t$. Assuming the walk starts at a vertex $v$, then $\mu_0(v) = 1$ and the rest entries are all 0. Let $X_t = u$ and at $t+1$, the walk stays at $u$ with probability $\frac12$ or jumps to a random neighbor with probability $\frac{1}{2\deg(u)}$ where $\deg(u)$ is the degree of $u$. The transition probabilities can be represented by a transition matrix $P$ where $P(i,j) = \Pr{X_{t+1} = j \mid X_{t} = i}$. Hence $\mu_{t+1} = \mu_t P$. By standard Markov chain theory, $\mu_t$ converges to a stationary distribution $\pi$ where $\pi(u) = \frac{\deg(u)}{2|E|}$. We use $\pi_t$ to represent the stationary distribution of the random walk using $G_t$ as the underlying static graph. Also we denote by $\deg_t(u)$ the degree of $u$ in $G_t$ and $P_t$ the transition probability matrix for $G_t$.
	In this section we introduce the relevant notation and basic results about Markov chains that we will use throughout the paper. For more background on Markov chains and random walks we defer the reader to~\cite{levin2017markov}.
	
	Let  $\sg = (G_t)_{t \in \mathbb{N}}$ be a sequence of undirected and unweighted graphs defined on the same vertex set $V$, with $|V| = n$, but with potentially different edge-sets $E_t$ ($t \in \mathbb{N}$). 
	We study (lazy) random walks on $\sg$: suppose that at a time $t\ge 0$ a particle occupies a vertex $u \in V$. At step $t+1$ the particle will remain at the same vertex $u$ with probability $1/2$, or will move to a random neighbor of $u$ in $G_t$. In other words, it will perform a single random walk step according to a transition matrix $P_t$, which is the transition matrix of a lazy random walk on $G_t$: $P_t(u,u) =1/2 $, $P_t(u,v) = 1/(2\deg_t(u))$ (where $\deg_t(u)$ is the degree of $u$ in $G_t$) if there  is an edge between $u$ and $v$ in $G_t$ , or $P_t(u,v) = 0$ otherwise. 
	
	Given an initial probability distribution $\mu_0 \colon V \to [0,1]$, which is the distribution of the initial position of the walk, the $t$-step distribution of a random walk on $\sg$ is equal to $\mu_t = \mu_0 P_1 \cdot P_2 \cdot \ldots \cdot P_t$. In particular, we use $\mu_t^x$ to denote the $t$-step distribution of the random walk starting at a vertex $x \in V$. Hence $\mu_0^x(x) = 1$ and $\mu_0^x(y) = 0$ for $x \ne y \in V$.  
	Furthermore, we use $\pi_t$ to denote the probability distribution with entries equal to $\pi_t(x)= \deg_t(x)/(2|E_t|)$ for any $x \in V$. This distribution is stationary for $P_t$ (i.e, it satisfies $\pi_t P_t = \pi_t$) and, if $G_t$ is connected, it is the unique stationary distribution of $P_t$. If $G_t$ is disconnected, $P_t$ will have multiple stationary distribution. However, unless stated otherwise, we will consider only the ``canonical'' stationary distribution $\pi_t$.  Finally, while any individual $P_t$ is \emph{time-reversible} (it satisfies $\pi_t (x) P_t(x,y) = \pi_t(y) P_t(y,x)$ for any $x,y \in V$), a random walk on $\mathcal{G}$ may not.
	%\NOTE{T}{perhaps ``may not''. Also should we explain why this is the case?}.
	\footnote{For example, it might happen that $P_1 \cdots P_t(x,y) > 0$ while $P_1 \cdots P_t(y,x) = 0$. This cannot happen in the ``static'' case where $P_1 = \cdots = P_t = P$ with $P$ reversible.}
	
	Recall that if $P$ is a transition matrix of a reversible Markov chain, it has $n$ real eigenvalues, which we denote with $-1 \le \lambda_n(P) \le \cdots \le \lambda_1(P) =1$. If $P$ is the transition matrix of a lazy random walk on a graph $G$, it holds that $\lambda_n(P) \ge 0$. Moreover, $\lambda_1(P) < 1$ if and only if $G$ is connected

	For two probability distributions $f,g \colon V \to [0,1]$, the \emph{total variation distance} between $f$ and $g$ is defined as $\norm{f - g}{TV} := \frac{1}{2} \sum_{x \in V} \abs{f(x) - g(x)}$. We denote with $\norm{f}{2} = \left( \sum_{x \in V} f^2(x) \right)^{1/2}$ and  $\norm{f}{\infty} = \max_{x \in V} |f(x)|$ the standard $\ell_2$ and $\ell_{\infty}$ norms of $f$. Given a probability distribution $\pi \colon V \to \mathbb{R}_+$, we also define the 
	$\ell_2(\pi)$-norm as $\norm{f}{2, \pi} := \sqrt{ \sum_{x \in V} f^2(x) \pi(x)}$. %This norm is required to reason about the spectral properties of random walks. 
	By Jensen's inequality, it holds for any $f,g$ that $\norm{f - g}{TV} \le \norm{f-g}{2, \pi}$. The  lemma below relates the decrease in the distance to stationarity after one random walk step to the spectral properties of its transition matrix.
	
	\begin{lemma}[Lemma 1.13 in \cite{montenegro2006mathematical}, rephrased] \label{lem:spectral}
		Let $P$ be the transition matrix of a lazy random walk on a graph $G=(V,E)$ with stationary distribution $\pi$. Then, for any $f : V \to \Real $, we have that 
		\[
		\norm{\frac{fP}{\pi} - \mathbf{1}}{2,\pi}^2 \le \lambda_2(P)^2 \norm{\frac{f}{\pi} - \mathbf{1}}{2,\pi}^2.
		\]
	\end{lemma}
	
	In the lemma above and throughout the paper, a division between two functions is to be understood entry-wise, while $\mathbf{1}$ refers to a function always equal to one.
	An important quantity which can be used to obtain bounds on $\lambda_2(P)$ is the \emph{conductance} of $G$, which is defined as follows.
	
	\begin{definition}\label{def:conductance}
		The conductance of a non-empty set $S \subseteq V$ in a graph $G$ is defined as:
		\[
		\Phi_G(S) := \frac{|E(S, V \setminus S)|}{\vol(S)},
		\]
		where $\vol(S) := \sum_{x \in V} \deg(x)$ and $E(S, V \setminus S)$ is the set of edges between $S$ and $V \setminus S$. The conductance of the entire graph $G$ is defined as
		\[
		\Phi_G := \min_{\substack{S \subset V \colon \\ 1 \leq \vol(S) \le \vol(V)/2}} \frac{|E(S, V \setminus S)|}{\vol(S)}.
		\]
	\end{definition}
	
	The conductance of $G$ and the second largest eigenvalue of the transition matrix $P$ of a lazy random walk in $G$ are related by the so-called discrete Cheeger inequality~\cite{alonMilman}, which we state below.
	
	\begin{theorem}[Cheeger inequality]\label{thm:cheeger}
		Let $P$ be the transition matrix of a lazy random walk on a graph $G$. Then, it holds that
		\[
		1 - \lambda_2(P) \le \Phi_G \le 2\sqrt{1 - \lambda_2(P)}.
		\]
	\end{theorem}
	
	Finally, we use the notation $o_n(1)$ to denote any function $f: \Nat \to \mathbb{R}$ such that $\lim_{n \to +\infty} f(n) = 0$. We often drop the subscript $n$.
	
	\subsection{Dynamic graph models} 
	In this section we formally introduce the random models of (dynamic) graphs that are the focus of this work. We start by recalling the definition of the \erdosrenyi model of random (static) graphs. 
	\begin{definition}[\erdosrenyi model]
		$G = (V,E) \sim \cg(n, p)$ is a random graph such that $|V| = \{1,\dots,n\}$ and the ${n \choose 2}$ possible edges appear independently, each with probability $p$.
	\end{definition}
	
	% \begin{definition}[Evolving graph model]
	% 	An evolving graph model $\cg$ is a sequence of graphs $\cg := (G_t)_{t \in \mathbb{N}}$  with the same vertex set $V=V_1=V_2=\cdots$ but possibly different set of edges $E_1, E_2, \ldots$
	% \end{definition}
	
	We now introduce the \emph{edge-Markovian} model of dynamic random graphs, which has been studied both in the context of information spreading in networks \cite{Clementi2016, Clementi2011} and random walks~\cite{SpirakisCover18}. This model is the focus of our work.
	
	\begin{definition}[edge-Markovian model]\label{def:edgemarkovian}
		Given a starting graph $G_0$, we denote with $(G_t)_{t \in \Nat} \sim \cg(n,p,q)$ a sequence of graphs such that $G_t = (V,E_t)$, where $V = \{1,\dots,n\}$ and, for each $t \in \mathbb{N}$, any pair of distinct vertices $u, v \in V$ will be connected by an edge in $G_t$ independently at random with the following probability:
		\[
		\Pr{\{u,v\} \in E_{t+1} \mid G_t} = \begin{cases}
		1-q & \text{ if } \{u,v\} \in E_t \\
		p & \text{ if } \{u,v\} \not\in E_t.
		\end{cases}
		\]
	\end{definition}
	Notice that different choices of a starting graph $G_0$ will induce different probability distributions over $(G_t)_{t \in \Nat}$. In general, we try to study $\cg(n,p,q)$ by making the fewest possible assumptions on our choice of $G_0$. Moreover, as pointed out for example in~\cite{SpirakisCover18}, $(G_t)_{t \in \Nat} \sim \cg(n,p,q)$ converges to $\cg(n,\tilde{p})$ with $\tilde{p} = p/(p+q)$. We leave considerations about the speed of this convergence and how this affects our choice of $G_0$ to Appendix~\ref{sec:graphchain} and, in particular, Remark \ref{rem:graphchain} .
	
	\subsection{Mixing time of random walks on dynamic graphs}
	%The standard (lazy) random walk on an evolving graph model $\cg = (G_t)_{t \in \Nat}$ is an extension of the random walk on a static graph. The graph evolves based on the previous definition. If $X_t = u$, then at $t+1$ it either stays with probability $\frac12$ or jumps to a random neighbor with probability $\frac{1}{2\deg_t(u)}$.  \NOTE{L}{I think we can remove this: it is explained in the notation section}
	%Typical questions about random walks focus on the \emph{mixing time}, the time to converge to the stationary distribution. However in our case, the walk on $G_t$ has different stationary distributions since the graph is changing. Hence the classic mixing time definition is not applicable. Below we first give the classic mixing time definition because we will need it in our proofs. Then we give the mixing time notion for evolving graph models. They also serve as a comparison between the definitions of the mixing times for static and dynamic graphs. 
	One of the most studied quantities in the literature about time-homogeneous (i.e., static) Markov chains (random walks included) is the mixing time, i.e., the time it takes for the distribution of the chain to become close to stationarity. Formally, it is defined as follows.
	\begin{definition}[Mixing time for time-homogeneous Markov chains]\label{def:staticmixingtime}
		Let $\mu_t^x$ be the $t$-step distribution of a Markov chain with state space $V$ starting from $x \in V$. Let $\pi$ be its stationary distribution. For any $\epsilon>0$, the $\epsilon$-mixing time is defined as
		\[
		\tmix(\epsilon) := \min \{t \in \Nat: \max_{x \in V} \norm{\mu_t^x - \pi}{TV} \le \epsilon \}.
		\]
	\end{definition}
	A basic fact in random walk theory states that a lazy random walk on a connected undirected graph $G=(V,E)$ has always a finite mixing time. In particular, if $|V|=n$, $\tmix(1/4) = O(n^3)$. Moreover, considering a different $\epsilon$ does not significantly change the mixing time: for any $\epsilon > 0$, $\tmix(\epsilon) = O(\tmix(1/4) \log(1/\epsilon))$~(see, e.g., \cite{levin2017markov}). Also, it is a well-known fact that $\norm{\mu_t^x - \pi}{TV}$ is non-increasing.
	
	However, in the case of random walks on dynamic graphs, convergence to a time-invariant stationary distribution does not, in general, happen. For this reason, other works have studied alternative notions of mixing for dynamic graphs, such as merging~\cite{merging3}, which happens when a random walk ``forgets'' the vertex where it started. In this work, instead, we focus on a different approach that we believe best translates the classical notion of mixing from the static to the dynamic case. More precisely, let us consider a dynamic sequence of graphs $(G_t)_{t \in \Nat}$ with corresponding stationary distributions $(\pi_t)_{t \in \Nat}$. Our goal is to establish if there exists a time $t$ such that the distribution $\mu_t$ of the walk at time $t$ is close to $\pi_t$. Moreover, to make this notion of mixing useful in possible applications, we require that $\mu_s$ remains close to $\pi_s$ for a reasonably large number of steps $s \ge t$. Formally, we introduce the following definition of mixing time for dynamic graph sequences.
	
	\begin{definition}[Mixing time for dynamic graph sequences]\label{def:dynamicmixingtime}
		Let $\cg=(G_t)_{t \in \Nat}$ be a dynamic graph sequence on a vertex set $V$, $|V|=n$. The mixing time of a random walk in $\cg$ is defined as 
		\[
		\tmix\left(\cg \right) = \min \left\{t \in \Nat \colon \forall t \le s \le t+\sqrt{n}, \; \forall x \in V, \;\;  \norm{\mu^{x}_s - \pi_s}{TV} = o_n(1) \right\},
		\] %\NOTE{Luca}{I've chosen to ask the distance to be small for $n$ steps but this can be changed (e.g., $n^{0.1}$ steps should suffice but it's a bit ugly).}\NOTE{Luca}{I've also chosen $o_n(1)$ instead of something like $1/\log{n}$ because I find it cleaner.}
		where $\pi_s$ is the stationary distribution of a random walk in $G_s$, and $\mu^{x}_s$ is the $s$-step distribution of a random walk in $\cg$ that started from $x \in V$.
	\end{definition}
	
	First observe we require that the total variation distance between $\mu_s$ and $\pi_s$ goes to zero as the number of vertices increases.\footnote{We are implicitly assuming there is an infinite family of dynamic graph sequences with increasing $n$.} This is motivated by the fact that the distance to stationarity, unlike in the static case, might never drop beyond a certain threshold: for this reason, we explicitly require that such threshold becomes smaller and smaller with an increasing number of vertices. Secondly, we require that such distance remains small for $\sqrt{n}$ steps (recall $n$ is the number of vertices in the graph). This is due to the fact that, for all dynamic graph models we consider, we cannot hope for such distance to stay small arbitrarily long. However, we believe that $\sqrt{n}$ steps is a long enough period of time for mixing properties to be useful in applications.
	
	Since our goal is to study the mixing property of $\cg(n,p,q)$, we now introduce a definition of mixing time for edge-Markovian models that takes into account the probabilistic nature of such graph sequences. Essentially, we say that the mixing time of $\cg(n,p,q)$ is $t$ if a random walk on a dynamic sequence of graphs sampled from $\cg(n,p,q)$ mixes (according to the previous definition) in $t$ steps with high probability over the sampled dynamic graph sequence.
	\begin{definition}[Mixing time for edge-Markovian models]\label{def:edgemarkovianmixingtime}
		Given an edge-Markovian model $\cg(n,p,q)$, its mixing time is defined as 
		\[
		\tmix\left(\cg(n,p,q) \right) = \min \left\{t \in \Nat \colon \ \Psub{\cg \sim \cg(n,p,q)}{\tmix\left(\sg \right) \le t} \ge 1 - o_n(1) \right\}.
		\]
	\end{definition}
	
	Finally, we remark that, while in static graphs connectivity is a necessary prerequisite to mixing, random walks on sequences of disconnected dynamic graphs might nonetheless exhibit mixing properties. Examples of this behavior were studied in~\cite{SauerwaldZanetti}.

	\section{Results for the fast-changing case}\label{sec:dependentfast}

	\subsection{Negative result for mixing in the sparse and fast-changing case} \label{sec:fastsparse}
	In this section we consider random walks on sparse and fast-changing edge-Markovian graphs. In particular, we study $\cg(n,p,q)$ with $0 < q = \Omega(1)$ and $p = \Theta(\frac{1}{n})$. Since $\Omega(1)$, by Remark~\ref{rem:graphchain}, we can restrict ourselves to consider the case where $G_0 \sim \cg(n,\tilde{p})$ with $\tilde{p} = p/(p+q)$. We will show a negative result on the mixing of random walks in this regime: no matter how small is the distance between $\mu_t$ and $\pi_t$, the total variation distance will increase to a positive constant with constant nonzero probability.
	
	\fastsparse*
	The key idea behind this result is that due to the fast-changing nature, the degrees of nodes change rapidly. In particular, for a linear number of nodes $u$, there is at least one neighbor $v_{\min} \in \Gamma_t(u)$ whose degree may change from one constant in round $t$ to, basically any other constant (this also makes use of the assumption on $p$, ensuring that the graph is sparse). The proof then exploits that, due to the ``unpredictable'' nature of this change, the probability mass received by $v_{\min}$ in round $t+1$ is likely to cause a significant difference between $\mu_{t+1}(u)$ and $\pi_{t+1}(u)$. Since this holds for a linear number of nodes $u$, we obtain a sufficiently large lower bound on the total variation distance, and the theorem is established. 
	
	\begin{proof}[Proof of \thmref{fastsparse}]
		We will prove that no matter how small the distance between $\mu_t$ and $\pi_t$ is, the total variation distance between $\mu_{t+1}$ and $\pi_{t+1}$ will increase to a positive constant with constant nonzero probability. This will yield the theorem.
		
		In this regime, at a time $t$ the graph has converged to $\cg(n, \tilde{p})$, and when the graph has isolated nodes, the stationary distribution of the random walk on a disconnected graph is not unique. Recall that in this case, we choose $\pi_t(u) = \deg(u)/2|E_t|$ as the stationary distribution. 
	
		For any node $u$, we have
		\begin{align*}
		\abs{ \mu_{t+1}(u) - \pi_{t+1}(u)} &= \abs{ \left( \sum_{v \in \Gamma_{t+1}(u)}\pi_t(v) \cdot \frac{1}{\deg_{t+1}(v)} \right) - \frac{\deg_{t+1}(u)}{2|E_{t+1}|} } \\
		& = \abs{ \sum_{v \in \Gamma_{t+1}(u)} \left( \frac{\deg_t(v)}{2|E_t|}\frac{1}{\deg_{t+1}(v)} - \frac{1}{2|E_{t+1}|} \right) } 
		\end{align*}
		where by Chernoff bound, we know $|E_{t+1}|/|E_t|$ is close to 1 and both of them are $\Theta(n)$ in this regime. Hence we can use $|E| = \binom{n}{2} \tilde{p}$ to replace either of them to get an approximation. So the distance is contributed by $\deg_t(v)/\deg_{t+1}(v)$. W.l.o.g. in our setting we assume $p = \frac{1}{n}$ and $q = \frac12$.
		
		For a vertex $u \in V$, let $v_{\min} \in \Gamma_{t}(u)$ be the neighbor of $u$ with the smallest degree in $G_t$, and assume $u$ is a vertex so that $\deg_{t}(v_{\min}) \leq 20$.
		Let us define $\alpha:= \sum_{v \in \Gamma_{t+1}(u) \setminus \{v_{\min}\}}  \left( \frac{\deg_t(v)}{\deg_{t+1}(v)} - 1 \right)$, the contribution by all neighbors of $u$ instead of $v_{\min}$. Also let us define $\beta:= \sum_{ v \in \Gamma_{t+1}(u)} \mathbf{1}_{ \{ \{v,v_{\min} \} \in E_{t+1} \} }$, the number of neighbors of $v_{\min}$ in $\Gamma_{t+1}(u)$ and $\gamma := \sum_{ v \in V \setminus \Gamma_{t+1}(u)} \mathbf{1}_{ \{ \{v,v_{\min} \} \in E_{t+1} \} }$, the number of neighbors of $v_{\min}$ in $V \setminus \Gamma_{t+1}(u)$. Let us define the following events: 
		\begin{align*}
		\ca &:= \{ |E_{t+1}| = (1 \pm \sqrt{10 \log n/ n} ) |E| \} \cap \{ |E_{t}| = (1 \pm \sqrt{10 \log n / n}) |E| \}, \\ \cb &:= \left\{ \beta \leq 80 \right\}, \\
		\cc &:=
		\left\{ \abs{ \alpha + \left( \frac{\deg_{t}(v_{\min})}{\beta+ \gamma  } - 1 \right) } \geq c_4
		\right\}, \\
		\cd &:= \left\{ v_{\min} \in \Gamma_{t+1}(u) \right\},
		\end{align*}
		where $|E| = \binom{n}{2}\tilde{p}$ is the expected number of edges of a stationary graph, and $c_4 > 0$ is a proper constant that is defined later.
		
		Using Chernoff bound, we have that
		\[
			\Pr{\neg \ca} \le \exp\left( \frac{10 \log n |E|}{2n} \right) = O(n ^{-c}),
		\]
		for some constant $c$.
		
		Note that $\Pr{ \ca} \geq 1-o(1)$, $\Pr{ \cb } \geq 3/4$ and $\Pr{ \cd } = 1-q > 0$ which is also a constant probability in this regime. If the event $\mathcal{A}$ holds, then 
		\begin{align*}
		\lefteqn{ \abs{ \mu_{t+1}(u) - \pi_{t+1}(u)} } \\
		&= \sum_{v \in \Gamma_{t+1}(u)} \frac{1}{2|E_t|} \left( \frac{\deg_t(v)}{\deg_{t+1}(v)} - 1 \right) \pm O(\log^2 n/\sqrt{n}).
		\end{align*}
		Further, assuming that $\mathcal{D}$ holds, we can lower bound
		\[
		\sum_{v \in \Gamma_{t+1}(u) \setminus \{v_{\min}\}} \frac{1}{2|E_t|} \left( \frac{\deg_t(v)}{\deg_{t+1}(v)} - 1 \right)
		+ \frac{1}{2|E_t|} \left( \frac{\deg_t(v_{\min})}{\deg_{t+1}(v_{\min})} - 1 \right).
		\]
		Now let us expose all degrees for $v \in \Gamma_{t+1}(u) \setminus \{v_{\min}\}$, subject to event $\mathcal{B}$ holding. Hence the above formula is equal to
		\[
		\frac{1}{2|E_t|} \cdot \left( \alpha + \frac{\deg_{t}(v_{\min})}{\beta+ \gamma  } - 1  \right).
		\]
		Recall that $\alpha$ and $\beta \leq 80$ are arbitrary. Further, $\gamma$ is independent of $\alpha$ and $\beta$. Also for any constant $c_1 > 0$, $\Pr{ \gamma = c_1} > c_2$ for some other constant $c_2 > 0 $. Note in order for $|\mu_{t+1}(u)-\pi_{t+1}(u)|$ not to contribute significantly to the total variation distance,  we must have
		\[
		\alpha + \frac{\deg_{t}(v_{\min})}{\beta+ \gamma  } - 1 = 0,
		\]
		which is equivalent to 
		\[
		\gamma = \frac{\deg_{t}(v_{\min})}{1-\alpha} - \beta.
		\]
		There is at most one possible (positive integer) value for $\gamma$ that solves this equation after $\alpha, \beta$ have been revealed. In particular, there is at least one constant $c_3 > 0$ so that 
		\[
		\left| \alpha + \frac{\deg_{t}(v_{\min})}{\beta+ c_3  } - 1  \right| > c_4,
		\]
		for some other constant $c_4 > 0$ because $v_{\min} \le 20$ and $\beta \le 80$. This proves $\Pr{ \cc \mid  \cb} \geq c_5$. Further,
		\begin{align*}
		\Pr{ \mathcal{A}  \cap \mathcal{C}  } &\geq \Pr{ \mathcal{A} } - \Pr{ \neg \mathcal{C}} \\
		&\geq \Pr{ \mathcal{A}} - 1 + \Pr{  \mathcal{C}  } \\
		&\geq \Pr{ \mathcal{A}}  - 1 + \Pr{  \mathcal{C} \wedge \mathcal{B} } \\
		&\geq \Pr{  \mathcal{A}} - 1 + \Pr{  \mathcal{B}} \cdot \Pr{ \mathcal{C} \, \mid \, \mathcal{B} } \\
		&\geq -o(1) + \frac{3}{4} \cdot c_5 > 0.
		\end{align*}
		Hence with constant probability, any vertex $u \in V$ in $G_{t}$ that has a neighbor $v_{\min}$ with $\deg_t(v_{\min}) \leq 20 $ will contribute $\Omega(1/|E|)$ to the $\ell_1$-norm with constant probability $c_6 > 0$. Let us now lower bound the number of such vertices, $S_t :=\left\{ u \in V(G_t) \colon \min_{v \in \Gamma_t(u)} \{ \deg(v)  \} \leq 20 \right\}$. 
		
		Fix any vertex $u \in V$. With constant probability $C_1 > 0$, it has at least one neighbor, say, $w$. Further, that neighbor $w$ will have at most $19$ neighbors (besides $u$) with constant probability $C_2 > 0$. Hence $\Ex{|S_t|} \geq (C_1 C_2) \cdot n.$ Using Markov's inequality, it follows that
		\[
		\Pr{ n - |S_t| \geq (1+C_1 C_2/2) \cdot (1-C_1 C_2) \cdot n} \leq \frac{1}{1+C_1 C_2/2}.
		\]
		Rearranging the above, it follows that with constant probability $C_3:=\frac{1}{1+C_1 C_2/2} > 0$, we have $|S_t| \geq (C_1 C_2/2) \cdot n$.
		Hence the overall contribution of all vertices in $S_t$ is at least $\Omega(n) \cdot \Omega(1/|E|) = \Omega(1)$ (as $|E|=\Theta(n)$ in this regime), with some constant probability $c_7 > 0$. 
	\end{proof}

	\subsection{Positive result for mixing in the dense and fast-changing case} \label{sec:fastdense}
	In this section we analyse the mixing properties of $\cg(n,p,q)$ for $p = \Omega(\log{n}/n)$ and $q = \Omega(1)$. Since $q$ is large, for simplicity we will assume throughout this section that $G_0 \sim \cg(n,\tilde{p})$, where $\tilde{p} = \frac{p}{p+q}$ (see Remark~\ref{rem:graphchain} for an explanation of why this is not a restriction).
	The main result of this section is the following theorem.
	
	\fastdense*
	
	% \begin{theorem} \label{thm:densefast}
	% Let $p = \omega\left(\log{n}/n\right)$ and $q = \Omega(1)$. Then,
	% $\tmix(\cg(n,p,q)) = O(\log{n})$.
	% \end{theorem}
	
	While in this paper we study for simplicity only lazy random walks on graphs, to prove \thmref{fastdense}, however, we need to introduce \emph{simple} random walks on graphs: given a graph $G=(V,E)$, a simple random walk on $G$ has transition matrix $Q$ such that, for any $x,y \in V$, $Q(x,y) = 1/\deg(x)$ if $\{x,y\} \in E$,  $Q(x,y) = 0$ otherwise. %We exploit the strong expansion properties of simple random walks on rand, which lazy random walks lack, in the following lemma, which is the main technical argument of this section. 
	The following lemma, whose proof is the main technical part of the section, shows that if the \emph{simple} random walk on a sequence of graphs $\cg = (G_t)_{t \in \mathbb{N}}$ exhibits strong expansion properties, and the time-varying stationary distribution is always close to uniform, then a \emph{lazy} random walk on $\cg$ will be close to the stationary distribution of $G_t$ for any $t$ large enough. We remark that a strong expansion condition on lazy random walks can never be satisfied; luckily, we just need this strong expansion condition to hold for their simple counterpart. %The proof of this lemma can be found in the Appendix.
	
	%\textcolor{blue}{Luca: I need to check constants}
	\begin{lemma} \label{lem:fastdense}
		Let $(G_t)_{t \in \mathbb{N}}$ be a sequence of graphs, and $(P_t)_{t \in \mathbb{N}}$ (resp. $(Q_t)_{t \in \mathbb{N}}$) the corresponding sequence of transition matrices for a \emph{lazy} (resp. \emph{non-lazy}) random walk. Assume there exists  $1<C=O(1)$ such that, for any $t \geq 1$ and any $x \in V$, $1/(C \cdot n) \le \pi_t(x) \le C/n$. Moreover, also assume that, for any $t \in \Nat$, $\max\{|\lambda_2(Q_t)|,|\lambda_n(Q_t)|\} \le \lambda = o(1)$. Then, there exists an absolute constant $C'$ such that, w.h.p., for any $t \ge C'\log{n}$ and any starting distribution $\mu_0$,
		\[
		\left\| \frac{\mu_{t}}{\pi_{t}} - \mathbf{1} \right\|_{2,\pi_{t}}^2 \le 10C^2(C-1)^2,
		\]
		where $\mu_t = \mu_0 P_1 \cdots P_t$.%\NOTE{T}{Perhaps it makes sense to define $p_t$ in the notation section, especially if we use it multiple times in the paper.}\NOTE{L}{I changed it to $\mu_t$ because that's what we used in the notation. However, $p_t$ appears multiple other times in the paper.}
	\end{lemma}
	
	\begin{proof}[Proof of \lemref{fastdense}]
		We first relate the $\ell_2(\pi)$ distance to stationarity to the $\ell_2$ distance from the uniform distribution. We start by observing that, for any $t \in \Nat$, by our assumptions on $\pi_t$, it holds that
		\begin{align*}
		\|\pi_t - \mathbf{1/n} \|_2^2 &= \sum_{x \in V} \left( \pi_t(x) - \frac{1}{n} \right)^2 \\
		&\le \max\left\{\left(1 - \frac{1}{C}\right)^2, (C-1)^2 \right\} \cdot \frac{n}{n^2} \\
		&= \frac{(C-1)^2}{n}.
		\end{align*}
		
		Then, for any probability distribution $p$, we have that
		\begin{align}
		\left\| \frac{p}{\pi_t} - \mathbf{1} \right\|_{2,\pi_t}^2 &= 
		\sum_x \frac{(p(x)-\pi_t(x))^2}{\pi_t(x)} \nonumber \\
		&\le Cn \cdot \| p - \pi_t \|_2^2 \nonumber \\
		&\le 2Cn \cdot \left(\| p - \mathbf{1/n} \|_2^2 + \|\pi_t - \mathbf{1/n} \|_2^2  \right) \label{eq:trianglesquare} \\
		&\le 2Cn \cdot \left(\| p - \mathbf{1/n} \|_2^2 + (C-1)^2/n  \right) \nonumber \\
		&= 2Cn \cdot \| p - \mathbf{1/n} \|_2^2 + 2C(C-1)^2, \label{eq:ell2norm}
		\end{align}
		where \eqref{eq:trianglesquare} holds by the triangle inequality and the basic inequality $(a+b)^2 \le 2a^2 + 2b^2$.
		Analogously, it holds that
		\begin{align}
		\left\| \frac{p}{\pi_t} - \mathbf{1} \right\|_{2,\pi_t}^2 
		&\ge (n/C) \cdot \left(\frac{1}{2} \| p - \mathbf{1/n} \|_2^2 - 3 \|\pi_t - \mathbf{1/n} \|_2^2 \right) \label{eq:trianglesquare2} \\
		&\ge (n/C) \cdot \left(\frac{1}{2} \| p - \mathbf{1/n} \|_2^2 - \frac{3}{n}(C-1)^2 \right) \nonumber \\ 
		&\ge (n/C) \cdot \| p - \mathbf{1/n} \|_2^2 - 3(C-1), \label{eq:ell2norm2}
		\end{align}
		where \eqref{eq:trianglesquare2} holds by the triangle inequality and the basic inequality $(a-b)^2 \ge a^2/2 - 3b^2$.
		
		Notice that the distance to the uniform distribution does not change if at step $t$ we perform a lazy step, which happens with probability $1/2$. Conditioning on the fact that we don't take a lazy step, at time $t$ we can bound the decrease in the distance to the uniform distribution as follows:
		\begin{align}
		\left\| \frac{\mu_{t+1}}{\pi_{t+1}} - \mathbf{1} \right\|_{2,\pi_{t+1}}^2 
		&\le \lambda \cdot \left\| \frac{\mu_{t}}{\pi_{t+1}} - \mathbf{1} \right\|_{2,\pi_{t+1}}^2 \nonumber \\
		&\le \lambda \cdot \left(2Cn \cdot \| \mu_{t} - \mathbf{1/n} \|_2^2 + 2C(C-1)^2\right), \label{eq:bigdecrease}
		\end{align}
		where the first inequality follows from \lemref{spectral} and the second by \eqref{eq:ell2norm}. Moreover, we have that
		\begin{align}
		n \cdot \| \mu_{t+1} - \mathbf{1/n} \|_2^2 
		&\le C \left\|\frac{\mu_{t+1}}{\pi_{t+1}} - \mathbf{1} \right\|_{2,\pi_{t+1}}^2 + 3C(C-1) \nonumber \\ 
		&\le o(\| \mu_{t} - \mathbf{1/n} \|_2^2 + C^2(C-1)^2)+ 3C(C-1), 
		\label{eq:decrease}
		\end{align}
		where the first inequality follows from \eqref{eq:ell2norm2}, and the last from \eqref{eq:bigdecrease} and the assumption $\lambda = o(1)$.
		
		This implies that whenever $n \cdot \| \mu_{t} - \mathbf{1/n} \|_2^2$ is large enough (e.g., $n \cdot \| \mu_{t} - \mathbf{1/n} \|_2^2 \ge 4C(C-1)$), if we condition on the walk not taking a lazy step at time $t$, the distance to the uniform distribution will shrink significantly (this follows because $\lambda = o(1)$). Therefore, we just need $O(\log{n})$ non-lazy steps for such distance to become small. Hence, after $t=O(\log{n})$ steps, it holds w.h.p. that
		\begin{equation} \label{eq:smalldist}
		n \cdot \| \mu_{t} - \mathbf{1/n} \|_2^2 \le 4C(C-1),
		\end{equation}
		which also implies by \eqref{eq:ell2norm} that
		\[
		\left\| \frac{\mu_{t}}{\pi_{t}} - \mathbf{1} \right\|_{2,\pi_{t}}^2 \le 10C^2(C-1)^2.
		\]
		
		Moreover, after $O(\log(n))$ steps, this distance will continue to be small. In fact, let $\mu_{t}$ satisfy \eqref{eq:smalldist}. If we condition on taking a lazy step at time $t+1$, such distance will not change. If instead we take a non-lazy step, by \eqref{eq:decrease}, 
		\[
		n \cdot \| \mu_{t} - \mathbf{1/n} \|_2^2 \le 3C(C-1) + o(C^2(C-1)^2) \le 4C(C-1),
		\]
		and, therefore, the distance to the uniform distribution again satisfies \eqref{eq:smalldist}. The lemma follows by applying \eqref{eq:ell2norm} once again.
	\end{proof}
	
	We now show how it can be used to derive \thmref{fastdense}. First recall that since we are assuming $G_0 \sim \cg(n,\tilde{p})$, all graphs in the sequence $(G_t)_{t \in \mathbb{N}}$ are sampled (non-independently) from $\cg(n,\tilde{p})$ (see Appendix \ref{sec:graphchain}). Furthermore, for any $t\in \Nat$, the assumptions of \thmref{fastdense} on $\lambda_2(Q_t)$ and $\lambda_n(Q_t)$ are satisfied with probability $1-o(1/n^2)$ for any graph sampled from $\mathcal{G}(n,\tilde{p})$ with $\tilde{p} > 2\log{n}/n$ by \cite[Theorem~1.1]{hoffman12}. Moreover, for  $\tilde{p} = \omega\left(\log{n}/n\right)$, by standard Chernoff bounds argument we can show that, with probability $1-o(1/n^2)$, all vertices of a graph sampled from $\cg(n,\tilde{p})$ have degree $(1+o_n(1))n\tilde{p}$. This implies that, for any $t$, w.h.p, the stationary distribution of $G_t$ satisfies the assumptions of \lemref{fastdense} with $C= 1 + o(1)$, which yields \thmref{fastdense}.
	
	It is natural to ask if we can relax the condition on $p$. Assume for example that $p,q$ are such that $\tilde{p} = p /(p+q) > 2\log{n}$. By \cite[Theorem~1.1]{hoffman12}, the conditions on $\lambda$ are still satisfied. However, it only holds that $C=\Theta(1)$. Therefore, \lemref{fastdense} can only establish that the $\ell_2(\pi_t)$-distance to stationarity is a constant (potentially larger than $1$). This, unfortunately, does not give us any meaningful bound on the total variation distance. However, if the $\ell_2$-distance between two distributions $\mu$ and $\pi$ is small, $\mu(x)$ cannot be much larger than $\pi(x)$. In a sense, this result can be interpreted as a \emph{coarse} mixing property. This is summarised in the following proposition.
	
	\begin{proposition}\label{prop:coarse}
		Let $(G_t)_{t \in \mathbb{N}} \sim \cg(n,p,q)$ with $p/(p+q) > 2 \log{n}/n$ and $q = \Omega(1)$. Let $\pi_t$ be the stationary distribution of $G_t$. Then, there exists absolute constants $c_1,c_2>0$ such that, for any starting distribution $\mu_0$ and any $c_1 \log{n} \le t \le \sqrt{n} + c_1 \log{n}$, it holds that
		\[
		\Pr{\left\| \frac{\mu_{t}}{\pi_{t}} - \mathbf{1} \right\|_{2,\pi_{t}}^2 \le c_2} \ge 1 - o_n(1). 
		\]
	\end{proposition}

	\section{Results for the slowly-changing case}\label{sec:dependentslow}
	
	%\NOTE{T}{Some of the lemmas and theorems in this section are quite long. Maybe we could say at the beginning of a (sub-)section that all statements in this (sub-)section assume the dense (or sparse) and slowly changing case, and then we don't need to repeat this in the statements.}

	\subsection{Positive result for mixing in the dense and slowly-changing case}\label{sec:slowdense}
	The aim of this section is to prove the following theorem.
	\slowdense*
	
	We start by proving that, if the three assumptions of \thmref{slowdense} are satisfied, then, for any $t=O(nd\log n)$, the conductance of $G_t$ is not much worse than the conductance of $G_0$ (with high probability). 
	\begin{lemma}[Conductance lower bound]\label{lem:conductancelowerbound}
		Let $d=\Omega(\log n)$, $p=O(\log n/n^2)$,  and $q=O(\log n/(dn))$. Assume that $G_0$ satisfies assumptions (1),(2),(3) of \thmref{slowdense}.
		Then,  there exists a constant $c>0$ such that, for any $t = O(n d\log n)$ and any vertex $v \in V$,
		\[
		\Pr{\deg_t(v) \le \frac{1}{2} \deg_0(v) } = O(n^{-4})
		\]
		and 
		\[
		\Pr{ \Phi_{G_t} \geq c \cdot \Phi_{G_0} } = 1- O(n^{-4}).
		\]
	\end{lemma}
	The proof of this lemma proceeds as follows: for any $S \subset V$, when an edge is randomly added or removed from the graph, we show that the probability that $|E_t(S,V \setminus S)|$ increases is usually larger than the probability it decreases. Therefore, we model $|E_t(S,V \setminus S)|$ as a random walk on $\Nat$ with a bias towards large values of $|E_t(S,V \setminus S)|$, i.e., a \emph{birth-and-death} chain. Using standard argument about birth-and-death chains, we show in \lemref{conductancelowerbound} that it is very unlikely that $|E_t(S,V \setminus S)|$ becomes much smaller than $|E_0(S,V \setminus S)|$. By a similar argument, in \lemref{upperboundvolume} we also show that the degrees of all nodes in $S$ are approximately the same as their original degrees in $G_0$. This ensures that the conductance of a single set $S$ is preserved after  $t=O(d n \log n)$ steps. We then use a union bound argument to show that, with high probability, the conductance of the entire graph is preserved. For certain value of $d$, however, we cannot afford to use an union bound an \emph{all} the possible sets of vertices. To overcome this, we show that we need to apply this union bound only for connected sets $S$. By carefully bounding the total number of such sets with respect to the maximum degree in $G_0$, we are able to establish the lemma.
	
	We can now give an outline of the proof of \thmref{slowdense}. The idea is to show that $\norm{\frac{\mu_{t+1}}{\pi_{t+1}} - \mathbf{1}}{2,\pi_{t+1}}$ is smaller than $\norm{\frac{\mu_t}{\pi_t} - \mathbf{1}}{2,\pi_t}$ (unless the latter is already very small). We do this by first relating $\norm{\frac{\mu_t}{\pi_t} - \mathbf{1}}{2,\pi_t}$ with $\norm{\frac{\mu_{t+1}}{\pi_t} - \mathbf{1}}{2,\pi_t}$. More precisely, we can use \lemref{spectral} and \lemref{conductancelowerbound} to show that the latter is smaller than the former by a multiplicative factor that depends on $\Phi_{G_0}$. Then, we bound the difference between $\norm{\frac{\mu_{t+1}}{\pi_t} - \mathbf{1}}{2,\pi_t}$ and $\norm{\frac{\mu_{t+1}}{\pi_{t+1}} - \mathbf{1}}{2,\pi_{t+1}}$. In particular, by exploiting the fact that at each step only $O(\log n)$ random edges can be deleted with high probability, we are able to show that $\norm{\frac{\mu_{t+1}}{\pi_{t+1}} - \mathbf{1}}{2,\pi_{t+1}}$ is not much larger than $\norm{\frac{\mu_{t+1}}{\pi_t} - \mathbf{1}}{2,\pi_t}$. Finally, by putting together all these argument, we show that $\norm{\frac{\mu_t}{\pi_t} - \mathbf{1}}{2,\pi_t}$ is monotonically decreasing in $t$, at least until the walk is mixed. This establishes the theorem.
	
	In the following analysis we need a standard Markov chain called \emph{birth-and-death} chain \cite[Chapter 2]{levin2017markov}. It is a random walk $(Z_t)_{t \in \Nat}$ on $\Nat$ whose transition probabilities are position based. Formally, we define
	\begin{equation}\label{eq:bdchaindefn}
	\Delta Z_t : = Z_{t+1} - Z_t =  
	\begin{cases}
	+1 & w.p.\ b_{Z_t}  \\
	0 & w.p. \ r_{Z_t} = 1 - b_{Z_t} - d_{Z_t}\\
	-1 & w.p.\ d_{Z_t} 
	\end{cases},
	\end{equation}
	where $(b_{Z_t}, r_{Z_t}, d_{Z_t})$ are the probabilities that the position of the walk increases by 1, remains, and decreases by 1. $b_{Z_t}$ and $d_{Z_t}$ are functions depending on the position $Z_t$ of the walk. Furthermore, we also need the \emph{hitting time} of a random walk $\tau^a_b := \min\{t \in \Nat: Z_t = b, Z_0 = a \}$ be the time the walk requires to hit $b$ when starting from $Z_0 = a$.  
	
	For a birth-and-death chain moving between $0$ and $m$, the stationary distribution of this Markov chain is given as follows:
	\begin{equation}\label{eq:bdchainstationary}
	\pi(k) = \frac{w_k}{\sum_{j = 0}^{m} w_j},
	\end{equation}
	where
	$
	w_k = \prod_{i = 1}^{k} \frac{b_{i-1}}{d_i}
	$
	and if all $b_{Z_t} = b'$ and $d_{Z_t} = d'$, $w_k = (b'/d')^k$. The birth-and-death Markov chain is \emph{time-reversible} \cite{levin2017markov} as defined in Sec.~\ref{sec:notations}. Hence $\pi(x)P^t(x,y) = \pi(y)P^t(y,x)$ where $P$ is the transition matrix of the Markov chain and $P^t(x,y)$ is the $t$-step transition probability from $x$ to $y$.
	
	To prove Lemma~\ref{lem:conductancelowerbound}, we first lower bound the number of edges in $E_t(S, V \setminus S)$ and then upper bound the volume of $S$ in $G_t$ for any $t = O(nd\log n)$. The following lemma proves that the number of edges on the boundary of any set $S$ will not decrease to $\epsilon$ less than the initial number with high probability. 
	
	\begin{lemma}[Lower bounding $|E_t(S, V \setminus S)|$]\label{lem:lowerboundboundary}
		Given the assumptions in Lemma \ref{lem:conductancelowerbound}, for one set $S$, let $M_t = |E_t(S, V \setminus S)|$ be the number of edges on the boundary of $S$ in $G_t$. Let $M_0 = m_0$ be the initial number. Denote $\Pr{\tau^a_b \le t }$ the probability that a walk starting from $a$ hits $b$ before $t$. 
		Then for any constant $\epsilon \in (0,1)$, we have 
		\[
		\Pr{ \tau^{m_0}_{m_T} = O(n d \log n) } = \begin{cases}
		O(n^{-c_1|S|}) & if\,|S| = O(\log n) \\
		O\left( \frac{\log d}{d} \right)^{-c_2|S|} & if\,|S| = \omega(\log n)
		\end{cases},
		\]
		where $m_T = \lfloor (1 - \epsilon)m_0 \rfloor$ and $c_1, c_2 \ge 4$ are constants. 
	\end{lemma}
	
	\begin{proof}
		By the settings of the regime and a Chernoff bound, the number of total changes in the graph is upper bounded by $O(\log n)$, and this holds for each of the first $n^3$ steps with probability $1 - n^{-\Omega(1)}$. Hence in one graph changing step, the quantity $|E_t(S, V \setminus S)|$ changes by at most $O(\log n)$. In the following analysis, we condition on this event being true. 
		
		We study the changes of the number of edges in $E_t(S, V \setminus S)$ by modeling it as a birth-and-death Markov chain. As mentioned above, from $M_t$ to $M_{t+1}$, there are at most $O(\log n)$ possible modifications. If we ``unpack'' them into single changes, then each change can be regarded as a random change on the number of edges in $E(S, V \setminus S)$. Each change on an edge slot can add an edge if there is no edge, remove an edge if there is an edge, or keep it as it is. Hence we build a birth-and-death chain $(M'_s)_{s \in \Nat}$ where $M'_0 = M_0 = m_0 = m'_0$ and $m_T = m'_S = \lfloor (1-\epsilon)m_0 \rfloor$. $M'_s$ is still the number of edges in $E(S, V \setminus S)$, but the time stamp represents the number of random changes we apply on the graph. From $M_t$ to $M_{t+1}$, the graph changing step can contribute $O(\log n)$ such random changes for $(M'_s)_{s \in \Nat}$. We denote $s(t)$ be the total number of random changes before time $t$ (included). 
		
		Hence the transition probability of this birth-and-death chain is then defined as: for all $s \in [s(t), s(t+1)]$,
		\[
		\Delta M'_s : = M'_{s+1} - M'_s =  
		\begin{cases}
		+1 & w.p.\ b_{M'_s} \ge \left(1 - \frac{M_t + O(\log n)}{|S|(n - |S|)} \right) \cdot p \\
		0 & w.p. \ r_{M'_s} = 1 - b_{M'_s} - d_{M'_s}\\
		-1 & w.p.\ d_{M'_s} \le \frac{M_t + O(\log n)}{|S|(n - |S|)} \cdot q
		\end{cases}
		\]
		When the chain is at the position $M'_s$ and one random change happens, the probability that it causes an edge addition is $\left(1 - \frac{M'_s}{|S|(n - |S|)} \right) \cdot p$ because it first needs to choose an empty slot in $E(S, V \setminus S)$ with probability $\left(1 - \frac{M'_s}{|S|(n - |S|)} \right)$ and then with probability $p$ it adds an edge. However note that from $M_t$ to $M_{t+1}$, the random changes are not independent because if one random change picks one edge slot, the other changes caused by the same graph changing step cannot choose the same edge slot again. Hence the actual number of empty slot is in fact lower bounded by $|S|(n - |S|) - M_t - O(\log n)$. Hence the probability $b_{M'_s}$ has the lower bound as shown above. This argument works similarly for $d_{M'_s}$. Due to the assumption that the initial graph has degree $\Omega(\log n)$, these bounds would still remain to be of the same order.
		
		Recall that $p = \Omega\left(\frac{\log n}{n^2} \right)$ and $q = O\left( \frac{\log n}{|E|} \right)$, hence the ratio is 
		\[
		\frac{b_{M'_s}}{d_{M'_s}} \ge \frac{(|S|(n - |S|) - M_t + O(\log n)) \cdot p}{(M_t + O(\log n)) \cdot q} = \Omega(|S|).
		\]
		The current regime can give us appropriate constants such that $\frac{b_{M'_s}}{d_{M'_s}} \ge 4$. Even if at the beginning the ratio may not satisfy this condition, there always exists a constant threshold $\epsilon'$ such that when $M'_s$ falls below $(1 - \epsilon')m'_0$, $\frac{b_{M'_s}}{d_{M'_s}}$ becomes less than 4. Hence w.l.o.g. we can wait until then and assume we start with $\frac{b_{M'_s}}{d_{M'_s}} \ge 4$. Note that $m'_0$ depends on $|S|$ so we do the following case analysis. 
		
		\paragraph{For small sets, $|S| \le 100\log n$:} By the assumption that the conductance is a constant, we have $M_0 = \Omega(|S|\log n)$. By (\ref{eq:bdchainstationary}) $\pi(m'_S) = O(1/4^{\epsilon m'_0})$ and $\pi(m'_0) \ge 3/4$. By the definition of the reversibility of a Markov chain \cite{levin2017markov}: for any time $t$,
		\[
		\pi(m'_0)p_t(m'_0, m'_S) = \pi(m'_S) p_t(m'_S, m'_0).
		\]
		where $p_t(x,y)$ means the $t$-step transition probability from $x$ to $y$. Hence
		\[
		p_t(m'_0, m'_S) \le \pi(m'_S) \frac{p_t(m'_S, m'_0)}{\pi(m'_0)}  \le \frac43 \pi(m'_S).
		\]
		Then we plug in $\pi(m'_t) = O(1/4^{\epsilon m'_0})$,
		\begin{align*}
		\Pr{\tau^{m_0}_{m_T} = O(nd\log n) } &= \Pr{\tau^{m'_0}_{m'_S} = O(nd\log^2 n) }\\
		& = \sum_{i = 0}^{O(nd\log^2 n)} p_i(m'_0, m'_S) \\
		&\le S \cdot \frac43 \pi(m'_S) \\
		&= O(n^{-c_1|S|}),
		\end{align*}
		where $c_1 \ge 4$ is a constant.
		
		\paragraph{For larger sized sets, $|S| > 100 \log n$:} By our assumption that the conductance is $\Omega(\log d/d)$. We have $m'_0 = \Omega(|S|\log d)$. Then $\frac{b_{M'_s}}{d_{M'_s}}$ becomes $O(d/\log d)$. Hence $\pi(m'_S) = O((\log d/d)^{\epsilon m'_0})$ and $\pi(m'_0) \ge 1 - O(\log d/d)$. By applying a similar analysis we have
		\begin{align*}
		\Pr{\tau^{m'_0}_{m'_S} = O(n d \log^2 n) } &= \sum_{i = 0}^{O(n\log^3 n)} p_i(m'_0, m'_S)\\ 
		& \le T \cdot \frac{1}{1 - O(\log d/d)} \cdot \pi(m'_S) \\
		&= O((\log d/d)^{-c_2 |S|})
		\end{align*}
		where $c_2 \ge 4$ is a constant. 
	\end{proof}
	
	\begin{lemma}[Upper bounding the volume, $\vol(S)$]\label{lem:upperboundvolume}
		Given the assumption in Lemma \ref{lem:conductancelowerbound}, for one set $S$, let $N_t = \vol_t(S)$, the sum of the degrees of the vertices of $S$ in $G_t$. Let $N_0 = n_0$ be the initial volume. Denote by $\Pr{\tau_b^a \le t}$ the probability that a walk starting at $a$ hits $b$ before $t$. 
		Then for any constant $\delta \ge 0$, we have
		\[
		\Pr{\tau^{n_0}_{n_T} = O(n d \log n)} = O(n^{-c|S|}),
		\]
		where $n_T = \lceil (1 + \delta)n_0 \rceil$ and $c \ge 4$.
	\end{lemma}
	
	\begin{proof}
		Denote $N_t$ the volume of the set $S$ at time $t$ then we have another birth-and-death chain $(N'_s)_{s \in \Nat}$ similar to the previous proof where $N'_0 = N_0 = n_0$. Let $s(t)$ denote the number of random changes before time $t$ (included). For all $s \in [s(t), s(t+1)]$, 
		
		\[
		\Delta N'_s : = N'_{s+1} - N'_s =  
		\begin{cases}
		1 & w.p.\ b_{N'_s} \le \left( 1 - \frac{N_t - O(\log n)}{(n-1)|S|} \right) \cdot p \\
		0 & w.p.\ 1 - b_{N'_s} - d_{N'_s}\\
		-1 & w.p.\ d_{N'_s} \ge \frac{N_t - O(\log n)}{(n-1)|S|} \cdot q
		\end{cases}
		\]
		In general this proof is similar to the previous one. The ratio between $b$ and $d$ is
		\[
		\frac{b}{d} \le O\left( \frac{((n-1)|S| - N_t + O(\log n)) \cdot p}{(N_t - O(\log n)) \cdot q} \right).
		\]
		This ratio is always a constant in our regime. Similar to the previous proof, we may not have a good ratio at the beginning, but there is always a constant threshold $\delta'$ such that when $N'_s = (1+\delta')N'_0$ the ratio $b/d$ is less than 1. Then by applying the same birth-and-death chain technique, we prove the lemma. 
	\end{proof}
	
	Now by combining the above two lemmas we can prove Lemma \ref{lem:conductancelowerbound}.
	
	\begin{proof}[Proof of Lemma \ref{lem:conductancelowerbound}]
		We have showed that for one set $S$, within $O(n\log^2 n)$ steps, the number of the edges on the boundary will not be $\epsilon$ smaller and the volume will not be $\delta$ larger than they originally were with high probability. 
		
		We apply a union bound to bound the conductance $\Phi_G$ of the entire graph. By Lemma \ref{lem:conductanceconnect}, $\Phi_G$ is revealed by just looking at the connected sets. Let $\tilde{G}$ be the union of all graphs from $G_1,..., G_t$. Then all the possible connected sets that ever exist in $\cg$ can be found in $\tilde{G}$. The number of all connected sets for one certain graph $G_i$ is upper bounded by those in the union graph. By Lemma \ref{lem:connectedsets}, the number of connected sets of size $|S|$ is bounded by $n \cdot \Delta^{|S|}$ where $\Delta$ is the maximum degree in $\tilde{G}$. By applying the birth-and-death chain argument for $|S| = 1$, the maximum degree $\Delta$ of the union graph should be upper bounded by $O(d)$. Hence below we use $n \cdot O(d^{2|S|-2})$ for the union bound when needed.
		
		Denote $\ce_1 (S)$ the event that the number of the edges on the boundary of $S$ ever reaches the $\epsilon$ less than the beginning. By union bound and Lemma \ref{lem:lowerboundboundary}, the probability that there exists such a set is upper bounded by
		\begin{align*}
		\Pr{\bigcup_{S \subseteq V} \ce_1(S) } &\le \Pr{\bigcup_{\substack{S \subseteq V \colon \\ |S| = O(\log n)}} \ce(S) }\\
		& \qquad + \Pr{\bigcup_{\substack{S \subseteq V \colon \\ |S| \in [\omega(\log n), O(n/\log n)]}} \ce_1(S) }\\ 
		& \qquad  + \Pr{\bigcup_{\substack{S \subseteq V \colon \\ |S| \in [\omega(n/\log n), O(n)]}} \ce_1(S) } \allowdisplaybreaks \\
		&\le \sum_{|S| = O(\log n)} n \cdot \Delta^{|S|} \cdot O(n^{-c_1|S|}) \\ 
		& \qquad + \sum_{[\omega(\log n), O(n/\log n)]} n \cdot \Delta^{|S|}\cdot O\left( \left(\frac{\log d}{d} \right)^{-c_2|S|} \right) \\
		& \qquad + \sum_{[\omega(n/\log n), O(n)]} \binom{n}{|S|}O\left( \left(\frac{\log d}{d} \right)^{-c_2|S|} \right) \allowdisplaybreaks\\
		& \le O(n^{-c_3})
		\end{align*}
		where $c_1, c_2$ are the constants used in Lemma \ref{lem:lowerboundboundary} and $c_3 \ge 4$ is a constant. 
		
		Denote $\ce_2 (S)$ the event that the volume of $S$ reaches $\delta$ larger than the beginning. By union bound and Lemma \ref{lem:upperboundvolume}, the probability that there exists such a set is upper bounded by
		\begin{align*}
		\Pr{\bigcup_{S \subseteq V} \ce_2(S) } \le \sum_{S \subseteq V} \binom{n}{|S|} O(n^{-c|S|}) = O(n^{-c'})
		\end{align*}
		where $c$ is the constant used in Lemma \ref{lem:upperboundvolume} and $c' \ge 4$ is a constant. 
		
		By combining everything above, we can lower bound the conductance. If the initial conductance is lower bounded by $\phi$, i.e.,
		$
		\Phi^0_G(S) = \frac{X_0}{Y_0} \ge \phi 
		$
		and
		\[
		\Phi_G^{t}(S) \ge \frac{ (1 - \epsilon) X_0}{(1+\delta) Y_0} \ge \frac{ (1 - \epsilon) \phi}{(1 + \delta)},
		\]
		where $\epsilon, \delta$ are the constants used in Lemma \ref{lem:lowerboundboundary} and Lemma \ref{lem:upperboundvolume}. For simplicity and also because of the arguments we have made along this proof, we choose come constants to be 4 in our final statement of the lemma.
	\end{proof}
	
	\begin{proof}[Proof of Theorem \ref{thm:slowdense}]
		We establish the theorem by showing unless $\nor{\frac{\mu_t}{\pi_t} - \mathbf{1}}{2, \pi_t}$ is already small, $\nor{\frac{\mu_t}{\pi_t} - \mathbf{1}}{2, \pi_t}$ will significantly decrease at each step. In particular we relate $\nor{\frac{\mu_t}{\pi_t} - \mathbf{1}}{2, \pi_t}$ to $\nor{\frac{\mu_{t+1}}{\pi_{t+1}} - \mathbf{1}}{2, \pi_{t+1}}$ in two steps:
		\begin{enumerate}
			\item[(1)] We lower bound the change between $\nor{\frac{\mu_t}{\pi_t} - \mathbf{1}}{2, \pi_t}$ and $\nor{\frac{\mu_{t+1}}{\pi_t} - \mathbf{1}}{2, \pi_t}$; 
			\item[(2)] We upper bound the difference between $\nor{\frac{\mu_{t+1}}{\pi_t} - \mathbf{1}}{2, \pi_t}$ and $\nor{\frac{\mu_{t+1}}{\pi_{t+1}} - \mathbf{1}}{2, \pi_{t+1}}$.
		\end{enumerate}
		
		\paragraph{Step 1:} For the first step, we use a spectral argument. Using Lemma \ref{lem:spectral}:
		\[
		\norm{\frac{\mu_t}{\pi_t} - \mathbf{1}}{2, \pi_t}^2 - \norm{\frac{\mu_{t+1}}{\pi_t}^2 - \mathbf{1}}{2, \pi_t}^2 \ge (1 - \lambda_2^2(P_t))\norm{\frac{\mu_t}{\pi_t} - \mathbf{1}}{2, \pi_t}^2,
		\]
		by rearranging terms we get
		\[
		\norm{\frac{\mu_{t+1}}{\pi_t} - \mathbf{1}}{2, \pi_t}^2 \le \lambda_2^2(P_t) \norm{\frac{\mu_t}{\pi_t} - \mathbf{1}}{2, \pi_t}^2,
		\]
		where $\lambda_2(P_t)$ is the second largest eigenvalue of $P_t$, the transition matrix of $G_t$.
		%\NOTE{L}{Need to use consistent notation w.r.t. eigenvalues}
		
		\paragraph{Step 2:} Next step we upper bound the difference between $\nor{\frac{\mu_{t+1}}{\pi_t} - \mathbf{1}}{2, \pi_t}$ and $\nor{\frac{\mu_{t+1}}{\pi_{t+1}} - \mathbf{1}}{2, \pi_{t+1}}$. Due to the randomness of the graph we will compute the expectation of this difference. In the following analysis we condition on the event that at any time $t$, $|E_t| \in [(1-o(1))nd, (1+o(1))nd]$ where $d = (n-1)\tilde{p}$. This event has probability $1 - o(1)$. Recall that
		\[
		\norm{\frac{\mu_t}{\pi_t} - \mathbf{1}}{2, \pi_t}^2 = \sum_{y \in V} \pi_t(y) \left( \frac{\mu(y)}{\pi_t(y)} - 1 \right)^2 = \left( \sum_{y \in V} \frac{\mu^2_t(y)}{\pi_t(y)} \right) - 1. 
		\]
		Hence we have
		\begin{align}
		&\Ex{ \norm{\frac{\mu_{t+1}}{\pi_{t+1}} - \mathbf{1}}{2, \pi_{t+1}}^2  - \norm{\frac{\mu_{t+1}}{\pi_t} - \mathbf{1}}{2, \pi_t}^2 }  \allowdisplaybreaks \nonumber\\
		& = \Ex{ \left( \sum_{y \in V} \frac{\mu^2_{t+1}(y)}{\pi_{t+1}(y)} \right) - \left( \sum_{y \in V} \frac{\mu^2_{t+1}(y)}{\pi_t(y)} \right) } \allowdisplaybreaks \nonumber\\
		& = \sum_{y \in V} \Ex{ \mu^2_{t+1}(y)\left( \frac{1}{\pi_{t+1}(y)} - \frac{1}{\pi_t(y)} \right) }\allowdisplaybreaks \nonumber\\
		& = \sum_{y \in V} \Ex{ \mu^2_{t+1}(y)\left( \frac{2|E_{t+1}|}{\deg_{t+1}(y)} - \frac{2|E_t|}{\deg_t(y)} \right) }\allowdisplaybreaks \nonumber\\
		& \le 2(1+o(1))|E| \sum_{y \in V} \Ex{ \mu^2_{t+1}(y) \left( \frac{1}{\deg_{t+1}(y)} - \frac{1}{\deg_t(y)} \right) }\allowdisplaybreaks \nonumber\\
		& \le 2(1+o(1))|E| \sum_{y \in V}  \mu^2_{t+1}(y) \Ex{\left( \frac{1}{\deg_{t+1}(y)} - \frac{1}{\deg_t(y)} \right) }\allowdisplaybreaks \label{line:1}\\
		& \le 2(1+o(1))|E| \sum_{y \in V}  \mu^2_{t+1}(y) \frac{\epsilon \deg_t(y)}{\deg_t(y)\cdot(1-\epsilon)\deg_t(y)}(1 - (1-q)^{\deg_t(y)})\allowdisplaybreaks \label{line:2}\\
		& \le \frac{2(1+o(1))}{1 - o(1)} \sum_{y \in V} \frac{\epsilon }{(1-\epsilon)} \cdot \frac{\mu^2_{t+1}(y)}{\deg_t(y)/((1-o(1))|E|)}(1 - (1-q)^{\deg_t(y)}) \allowdisplaybreaks \nonumber\\
		& \le \frac{2(1+o(1))}{1 - o(1)} \cdot \frac{\epsilon }{(1-\epsilon)} (1 - (1-q)^{\deg_t(y)}) \sum_{y \in V}  \frac{\mu^2_{t+1}(y)}{\pi_t(y)} \allowdisplaybreaks \nonumber\\
		& \le O\left( \frac{\log n}{n} \right)\left(\norm{\frac{\mu_{t+1}}{\pi_t} - \mathbf{1}}{2, \pi_t}^2 + 1\right) \allowdisplaybreaks \label{line:3}
		\end{align}
		where $|E| = nd$ and $d = (n-1)\tilde{p}$. The $\epsilon$ is the constant used by Lemma \ref{lem:conductancelowerbound}. $\deg_{t+1}(y)$ will not decrease to $\epsilon$ smaller than $\deg_{t}(y)$ with probability $1 - O(n^{-c})$. From line (\ref{line:1}) to line (\ref{line:2}) we upper bound the expectation by only considering the cases where the difference is positive, i.e., $\deg_t(y) \ge \deg_{t+1}(y)$. In line (\ref{line:2}), by \lemref{conductancelowerbound} we know $\deg_{t+1}(y)$ will not be smaller than $\frac{1}{2} \cdot \deg_{t}(y)$ with probability $1 - O(n^{-4})$. Moreover, the probability $1  - (1 - q)^{\deg_t(y)}$ is the probability that at least one of the edges connected to $y$ at time $t$ changes at $t+1$. In line (\ref{line:3}), we hide unimportant constants in the $O$-notation and we use the inequality $(1-q)^{\deg_t(y)} \ge 1 - q \cdot \deg_t(y)$. Since $q = O(\log n/(dn))$ by assumption, we get $O(\log n/n)$ in line (\ref{line:3}). 
		
		By combining the two steps above we have
		\begin{align*}
		&\Ex{ \norm{\frac{\mu_{t+1}}{\pi_{t+1}} - \mathbf{1}}{2, \pi_{t+1}}^2  - \norm{\frac{\mu_t}{\pi_t} - \mathbf{1}}{2, \pi_t}^2 } \allowdisplaybreaks\\
		& \le O\left( \frac{\log n}{n} \right) \left(\norm{\frac{\mu_{t+1}}{\pi_t} - \mathbf{1}}{2, \pi_t}^2 + 1 \right)  - (1 - \lambda_2^2(P_t)) \norm{\frac{\mu_t}{\pi_t} - \mathbf{1}}{2, \pi_t}^2\allowdisplaybreaks\\
		& \le O\left( \frac{\log n}{n} \right) \left( \lambda_2^2(P_t) \norm{\frac{\mu_t}{\pi_t} - \mathbf{1}}{2, \pi_t}^2 + 1 \right)  - (1 - \lambda_2^2(P_t)) \norm{\frac{\mu_t}{\pi_t} - \mathbf{1}}{2, \pi_t}^2\allowdisplaybreaks\\
		& \le \left(\frac{n+\log n}{n}\cdot \lambda_2^2(P_t) - 1 \right) \norm{\frac{\mu_t}{\pi_t} - \mathbf{1}}{2, \pi_t}^2 + O\left(\frac{\log n}{n} \right)
		\end{align*}
		Therefore, it holds that
		\[
		\Ex{ \norm{\frac{\mu_{t+1}}{\pi_{t+1}} - \mathbf{1}}{2, \pi_{t+1}}^2 } \le \left(\frac{n + \log n}{n} \right) \lambda_2^2(P_t) \cdot   \norm{\frac{\mu_t}{\pi_t} - \mathbf{1}}{2, \pi_t}^2 + O\left( \frac{\log n}{n} \right).
		\]
		By Theorem \ref{thm:cheeger} and the laziness of the walk,
		\[
		\frac{\Phi_{G_t}^2}{2} \le 1 - \lambda_2(P_t) \le 2\Phi_{G_t}.
		\]
		Since we assume the conductance is lower bounded by $O(\log d/d)$, we have $\lambda_2(P_t) \le 1 - O(\log^2 d/d^2)$ and hence $((n + \log n)/n)\lambda_2^2(P_t) \le 1$. Therefore in expectation the $\ell_2$ distance shrinks by a factor less than 1 but with an additive term. %\NOTE{T:}{What does ``which shows our dynamic mixing can never be arbitrarily small''. I think it is only an obstacle of our proof, but it does not prove a lower bound.} \NOTE{Leran}{Right. I used a wrong way to describe this term, it only means the additive term prevent us from saying an arbitrarily small distance. Though now I am not sure if we should mention it or not.}
		By a similar analysis for the static graph in \cite{levin2017markov}, after $O(\log n/\Phi_{G_0}^2)$ rounds, the expected distance to $\pi_t$ is at most $O(\sqrt{\log n/n})$. By Lemma \ref{lem:conductancelowerbound}, we know this holds for $poly(n)$ time. Hence it suffices to apply Markov's inequality and union bound to show that the expected distance is small with probability $1 - O(n^{-c})$ on a polynomially long time interval as required in the mixing time notion.
		% 	Luca: I think we just show that after a certain time (something likelog(n)/φ20)the expected distance to stationarity is very small (something like 1/n) and wesay that in expectation it keeps being small. Then we say that if we care theexpectation  to  be  small  w.h.p  on  an  interval  it  suffices  to  apply  Markov  plusunion bound.
	\end{proof}

	\subsection{Negative result for mixing in the sparse and slowly changing case}
	
	\slowsparse*
	\begin{proof}
		Consider the graph $G_0 \sim \cg(n,\tilde{p})$. Notice that $\tilde{p}=o(\log{n}/n)$ is well below the connectivity threshold of \erdosrenyi random graphs. Therefore, with high probability, there is at least one isolated vertex in $G_0$; call this vertex $u$ and assume the random walk starts from that vertex. The probability that $u$ remains isolated in the steps $1,2,\ldots,t$ is at least
		\[
		(1-p)^{(n-1) \cdot t} \ge (1-O(1/n^2))^{(n-1) \cdot t} \geq 1-O(t/n).
		\]
		Therefore, with at least constant nonzero probability, there exists a constant $c>0$ such that, for any $t \le c\cdot n$, $\mu_t^u(u) = 1$. Since $\pi_t(u) = 0$, this implies that $\norm{\mu_t^u - \pi^t}{TV} = 1$.
		%Conditioning on this event, the $t$-step distribution $\mu_t$ satisfies $\mu_t(u)=1$ and $\mu_t(v)=0$ for $v \neq u$. However, $\pi_t(u)=1/n$, and thus
		% \begin{align*}
		%  \left\| \mu_t/ \pi_t - \mathbf{1} \right\|_{2,\pi_t}^2 &=  \sum_{v \in V}  \left( \mu_t(v)/\pi_t(v) - 1 \right)^2 \pi_t(v)
		%  \\
		%  &= (n-1)^2 \cdot 1/n + \sum_{v \neq u} 1/n = n^2 - 2 + 1/n + (n-1)/n = n^2 - 1.
		%\end{align*}
	\end{proof}
	
	Actually the proof reveals a stronger ``non-mixing'' property; if the random walk starts from a vertex that is isolated in $G_0$, then this vertex will remain isolated for $\Theta(1/(np))$ rounds in expectation, and in this case the random walk did not move at all! 
	
	%\NOTE{T}{We should perhaps refer to the precise definitions of strong mixing etc. in the next section.}

	\section{Conclusion}
	In this work we investigated the mixing time of random walks on the edge-Markovian random graph model. Our results cover a wide range of different densities and speeds by which the graph changes. On a high level, these findings provide some evidence of the intuition that both of the two properties ``high density'' and ``slow changes'' correlate with fast mixing. 
	
	For further work, one interesting setting that is not fully understood is the semi-sparse ($d=\Theta(\log n)$) and fast-changing ($q=\Omega(1)>0$) case. While we proved that the random walk achieves some coarse mixing in $O(\log n)$, we conjecture that strong mixing is not possible. Another possible direction for future work is, given the bounds on the mixing time at hand, to derive tight bounds on the cover time. Finally, it would be also interesting to study the mixing time in a dynamic random graph model, where not all edge slots are present (similar to the models studied in~\cite{SpirakisCover18,hermonSousi}, where the graph at each step is a random subgraph of a (possibly sparse) network).
	\bibliographystyle{acm}
	\bibliography{bib}

	\newpage
	\appendix
	
	\section{Mixing times for the graph chain of edge-Markovian models} \label{sec:graphchain}
	%\NOTE{Leran}{Explain why the mixing time of the graph chain is important.}
	It is well known that the edge-Markovian graph model $\cg(n,p,q)$ converges to an \erdosrenyi model $\cg(n, \tilde{p})$ where $\tilde{p} = \frac{p}{p+q}$, which is the stationary distribution of the original edge-Markovian model. The mixing time of the graph chain has not been proven formally in previous works. Hence, we provide a proof for the sake of completeness. We remark that since an edge-Markovian model is a time-homogeneous (i.e., static) Markov chain, the classical definition  of mixing time (Definition \ref{def:staticmixingtime}) applies.

	\begin{theorem}[Graph chain mixing time]\label{thm:graphmixing}
		For an edge-Markovian model $\cg(n,p,q)$, the graph distribution converges to the graph distribution of the random graph model $\cg(n, \tilde{p})$ where $\tilde{p} = \frac{p}{p+q}$. For any $\epsilon \in (0,1)$, the mixing time of the graph chain $\cg(n,p,q)$ is $\tmix(\epsilon) = O\left(\frac{\log (n/\epsilon) }{\log(1/|1 - p - q|)} \right)$ for $p+q \ne 1$, and $\tmix(\epsilon) = 1$ if $p+q = 1$.
	\end{theorem}
	%\NOTE{T}{Maybe move this theorem (and not just the proof) to appendix?}
	%\NOTE{L}{Problem with new definition of mixing time}
	
	\begin{proof}
		Every edge slot can be represented by a two-state (close/open) Markov chain with transition matrix
		\[
		P = \begin{pmatrix}
		1-p & p   \\
		q   & 1-q 
		\end{pmatrix}
		\]
		and stationary distribution $\left( \frac{q}{p+q}, \frac{p}{p+q} \right)$. By using standard Markov chain arguments (see, e.g., \cite[Chapter 1]{levin2017markov}), the distance to the stationary distribution shrinks at each step by a factor of $1 - p - q$, i.e., 
		\[
		\norm{\mu_{t+1} - \pi}{TV} \le |1 - p - q|\norm{\mu_t - \pi}{TV}.
		\]
		Therefore, when $p + q \ne 1$, the mixing time $\tmix(\epsilon)$ of this two-state Markov chain is $O\left(\frac{\log (1/\epsilon) }{\log|1 - p - q|} \right)$ where $\epsilon < 1$. For all the $\binom{n}{2}$ edge slots, the time that all of them mix is $O\left(\frac{\log \binom{n}{2} + \log (1/\epsilon) }{\log|1 - p - q|} \right)$. When $p + q = 1$, instead, the graph mixes immediately, which confirms the fact that in this regime the graph model is equivalent to a sequence of independent graphs from $\cg(n, \tilde{p})$.
	\end{proof}
	
	\begin{remark}\label{rem:graphchain}
		\thmref{graphmixing} essentially tells us that, whenever at least one between $p$ and $q$ is large (e.g., $\Omega(1)$), the graph chain quickly converges to $\cg(n, \tilde{p})$ with $\tilde{p} = \frac{p}{p+q}$. This suggests that for a fast-changing edge-Markovian model $\cg(n,p,q)$ with $q=\Omega(1)$, we can consider w.l.o.g. the starting graph $G_0$ as sampled from $\cg(n, \tilde{p})$.
	\end{remark}

	\section{Missing proofs in \secref{slowdense}}

	\begin{lemma}\label{lem:conductanceconnect}
		For any graph $G$, the conductance of $G$ is equal to  \[\Phi_G = \min_{\substack{S \subset V, \vol(S)\le\vol(G)/2 \\ \text{$S$ is connected} }} \Phi_G(S). \]
	\end{lemma}
	
	\begin{proof}
		By definition the conductance of a graph $\Phi_G$ is 
		\[
		\Phi_G = \min_{S \subseteq V, |S| \le n/2} \frac{|E(S, V \setminus S)|}{\vol(S)}
		\]
		
		Assuming the graph conductance is achieved by a disconnected set $D$. W.l.o.g. we assume it has two connected components $A$ and $B$. Then since there are no edges between $A$ and $B$, we have
		\[
		\Phi_G = \Phi_G(D) = \frac{|E(A, V \setminus A)| + |E(B, V \setminus B)|}{\vol(A) + \vol(B)}
		\]
		Then by a simple inequality,
		\[
		\Phi_G(D) \ge \min\left\{ \frac{|E(A, V \setminus A)|}{\vol(A)}, \frac{|E(B, V \setminus B)|}{\vol(B)} \right\}
		\]
		where $|A|, |B| \le |D| \le n/2$. By induction, this holds for all disconnected sets with more than two connected components. Hence if there exists a disconnected set $D$ which gives the minimum conductance over the entire graph then there must exist a connected set $S$ in the graph which also achieves this conductance. 
	\end{proof}
	
	\begin{lemma}\label{lem:connectedsets}
		Let $G$ be a graph with $n$ vertices and the maximum degree is $\Delta$. The number of all the connected sets with $k$ vertices is at most $n \cdot \Delta^{2k-2}$.
	\end{lemma}
	
	\begin{proof}
		For a fixed $k > 0$, we define an encoding which can represent a path of length $2k$ starting from a certain node. First we label $n$ vertices from $1$ to $n$. Then when enumerating all paths of length $2k$, we first output the label of the starting vertex, then for its neighbors we sort them based on their labels and then use the ranks in the following output.  
		
		The output should start with a label which represents the root node which is from 1 to $n$. Then the following numbers are all less than $\Delta$ because each of them only means the rank of a node among all the neighbors of its predecessors. Hence at most we have $n \cdot \Delta^{2k - 2}$ encodings. 
		
		There is an injective map between all the connected set of size $k$ with the encoding. For any such set, there is a spanning tree of size $k$. A DFS traversal of this tree would only use each edge of the spanning tree twice. Hence a not necessarily simple path of length $2(k-1)$ must exist. Our encoding essentially gives all possible paths of length $2(k - 1)$. So there must be such an injective map. The number of the strings in our output is $n \cdot \Delta^{2(k-1)}$. Hence the total number of connected sets is upper bounded by that as well. 
	\end{proof}

\end{document}

%% file: notationmacrosarxiv.tex
\newtheorem{theorem}{Theorem}[section]

\newtheorem{lemma}[theorem]{Lemma}
\newtheorem{definition}[theorem]{Definition}
\newtheorem{proposition}[theorem]{Proposition}
\newtheorem{remark}[theorem]{Remark}

\renewcommand{\deg}{\operatorname{deg}}

\newcommand{\Nat}{\mathbb{N}}

\newcommand{\Real}{\mathbb{R}}

\newcommand{\ce}{\mathcal{E}}
\newcommand{\cg}{\mathcal{G}}
\newcommand{\sg}{\mathscr{G}}
\newcommand{\ca}{\mathcal{A}}
\newcommand{\cb}{\mathcal{B}}
\newcommand{\cc}{\mathcal{C}}
\newcommand{\cd}{\mathcal{D}}

\newcommand{\Ex}[1]{\mathbb{E}\left[{#1}\right]}

\newcommand{\vol}{\mathrm{vol}}

\renewcommand{\Pr}[1]{\mathbb{P}\left[{#1}\right]}
\newcommand{\Psub}[2]{\mathbb{P}_{#1}\left[{#2}\right]}

\newcommand{\norm}[2]{\left\| {#1} \right\|_{#2}}
\newcommandtwoopt{\ip}[3][\cdot][\cdot]{\langle {#1}, {#2} \rangle_{{#3}}}
\newcommand{\abs}[1]{\left|{#1} \right|}

\newcommand{\tmix}{t_\mathrm{mix}}

\newcommand{\erdosrenyi}{Erd\H{o}s-R\'{e}nyi\xspace}

\usepackage{xspace,setspace,color}   % just for \NOTE
\newcounter{nummer}
\newcommand{\thmref}[1]{Theorem~\ref{thm:#1}}

\newcommand{\lemref}[1]{Lemma~\ref{lem:#1}}

\newcommand{\secref}[1]{Section~\ref{sec:#1}}